\documentclass[12pt]{amsart}

\usepackage{amssymb,amsmath}

\usepackage{graphicx}        
\usepackage{tikz}
\usetikzlibrary{backgrounds}
	
\usepackage{ifthen}
\usepackage{enumerate}

\usepackage[noadjust]{cite}
\usepackage{bbm, bm}

\newcommand{\dist}{{\rm dist}}
\newcommand{\forb}{{\rm Forb}}
\newcommand{\ed}{{\rm ed}}
\newcommand{\hh}{\mathcal{H}}

\newcommand{\Exp}{{\mathbb E}}
\newcommand{\aas}{{\bf a.a.s.}}

\newcommand{\bb}{\mathcal{B}}
\newcommand{\dd}{\mathcal{D}}
\newcommand{\ff}{\mathcal{F}}
\newcommand{\kk}{\mathcal{K}}

\newcommand{\GG}{\mathbb{G}}
\newcommand{\PP}{\mathbb{P}}
\newcommand{\RR}{\mathbb{R}}

\newcommand{\vk}{V(K)}
\newcommand{\vw}{{\rm VW}}
\newcommand{\vb}{{\rm VB}}
\newcommand{\ew}{{\rm EW}}
\newcommand{\eb}{{\rm EB}}
\newcommand{\eg}{{\rm EG}}
\newcommand{\vwk}{{\rm VW}(K)}
\newcommand{\vbk}{{\rm VB}(K)}
\newcommand{\ewk}{{\rm EW}(K)}
\newcommand{\ebk}{{\rm EB}(K)}
\newcommand{\egk}{{\rm EG}(K)}

\newcommand{\x}{{\bf x}}
\newcommand{\y}{{\bf y}}

\newcommand{\pt}{\tilde{p}}

\newtheorem{thm}{Theorem}
\newtheorem{lem}[thm]{Lemma}
\newtheorem{cor}[thm]{Corollary}
\newtheorem{prop}[thm]{Proposition}
\newtheorem{rem}[thm]{Remark}
\newtheorem{quest}[thm]{Question}
\newtheorem{conj}[thm]{Conjecture}

\newtheorem{defn}[thm]{Definition}

\newcommand\restr[2]{{
  \left.\kern-\nulldelimiterspace 
  #1 
  \vphantom{\big|} 
  \right|_{#2} 
  }}

\font\xxroman=cmr17 scaled \magstep1 \font\xviiroman=cmr17
\def\udot{\mathbin{\ooalign{$\cup$\crcr
   \hfil\raise 6pt\hbox{\xviiroman.}\hfil\crcr}}}
\def\bigudotx#1#2{\mathop{\smash{\ooalign{$#1\bigcup$\crcr
   \hfil\raise 6pt\hbox{#2}\hfil\crcr}}\vphantom{\bigcup}}}
\def\bigudot{\mathop{\mathchoice%
 {\bigudotx\displaystyle{$\scriptscriptstyle\bullet$}}
 {\bigudotx\textstyle{\xxroman.}}
 {\relax}{\relax}}}

\newcommand{\N}{\mathbb{N}}

\newcommand{\one}{{\bf 1}}
\newcommand{\zero}{{\bf 0}}
\newcommand{\bdelta}{{\bm \delta}}

\newcommand{\be}{{\bf e}}

\newcommand{\phirange}{1-\varphi^{-1},\varphi^{-1}}
\newcommand{\phirangeleft}{1-\varphi^{-1},1/2}
\newcommand{\phirangeright}{1/2,\varphi^{-1}}

\newcommand{\simplex}{\Delta}

\begin{document}

\title{On the edit distance function of the random graph}

\author{Ryan R. Martin and Alexander W.N. Riasanovsky}
\address{Department of Mathematics, Iowa State University, Ames, IA 50011-2064}
\email{rymartin@iastate.edu, awnr@iastate.edu}
\thanks{Both authors' research was partially supported by NSF award DMS-1839918 (RTG). Martin's was partially supported by Simons Foundation Collaboration Grant \#353292}
%
%

\begin{abstract}
    Given a hereditary property of graphs $\mathcal{H}$ and a $p\in [0,1]$, the edit distance function ${\rm ed}_{\mathcal{H}}(p)$ is asymptotically the maximum proportion of edge-additions plus edge-deletions applied to a graph of edge density $p$ sufficient to ensure that the resulting graph satisfies $\mathcal{H}$.  The edit distance function is directly related to other well-studied quantities such as the speed function for $\mathcal{H}$ and the $\mathcal{H}$-chromatic number of a random graph.
    
    Let $\mathcal{H}$ be the property of forbidding an Erd\H{o}s-R\'{e}nyi random graph $F\sim \mathbb{G}(n_0,p_0)$, and let $\varphi$ represent the golden ratio. In this paper, we show that if $p_0\in [1-1/\varphi,1/\varphi]$, then~\aas~as $n_0\to\infty$, 
	\begin{align*}
	    {\rm ed}_{\mathcal{H}}(p) = (1+o(1))\,\frac{2\log n_0}{n_0}
	    \cdot\min\left\{
	        \frac{p}{-\log(1-p_0)}, 
	        \frac{1-p}{-\log p_0}
        \right\}.  
	\end{align*}  
	Moreover, this holds for $p\in [1/3,2/3]$ for any $p_0\in (0,1)$.
\end{abstract}


\maketitle

\section{Introduction}
\label{sec:intro}

All graphs are finite and simple, i.e., without loops and multi-edges.  
A graph is {\em nonempty} if it has at least one edge.  
Denote $P_n$ to be the path graph on $n$ vertices. \\

For any $p\in [0,1]$ and any positive integer $n$, write $\GG(n,p)$ to be the distribution of graphs according to the Erd\H{o}s-R\'{e}nyi random graph model with edge probability $p$.  
That is, $G\sim\GG(n,p)$ means that the event that $uv\in E(G)$ for $uv\in \binom{\{1,\dots,n\}}{2}$ are independent and identically distributed (i.i.d.) with common probability $p$.  
We write \aas~ to mean that a sequence of events holds with probability approaching $1$ under some implied limit. The limit will be clear by the context.  
All logarithms are natural unless explicitly stated otherwise. \\

\subsection{Edit distance results and forbidding a random graph}\label{sub-sec:ed results}
The edit distance measures the minimum number of ``edits'' (that is, edge-additions plus edge-deletions) sufficient to turn one graph into another. This metric has been studied in contexts such as property testing and evolutionary biology (see~\cite{AlonStavFurthestGraph,MartinSurvey}).  
Formally, for any two $n$-vertex graphs $G,H$ on the same vertex set, 
\begin{align*}
    \dist(G,H) = |E(G)\triangle E(H)| \cdot {\textstyle \binom{n}{2}}^{-1}
\end{align*}
where $\triangle$ is the symmetric difference operation for sets. \\

A graph property $\hh$ is {\em hereditary} if $\hh$ is closed under isomorphism and vertex deletion.  
For any family $\ff$ of graphs, we may write $\forb(\ff)$ for the hereditary property of graphs which do not contain an induced copy of $F$ for any $F\in\ff$.  
Any hereditary property is of the form $\forb(\ff)$
for some $\ff$. 
A hereditary property of the form $\forb(\{F\})$ for a single graph $F$ is called a {\em principal hereditary property} and we will write $\forb(F)$ for simplicity. \\

A hereditary property $\hh$ is {\em nontrivial} if, for every positive integer $n$, there exists a graph in $\hh$ of order $n$.
All hereditary properties in this paper are nontrivial. 
If $\hh$ is a nontrivial hereditary property, then for all graphs $G$, we define 
\begin{align*}
    \dist(G,\hh) = \min\{\dist(G,H) : \exists H\in\hh\text{ s.t. }V(H) = V(G)\}.  
\end{align*}~\\

An early result that has motivated subsequent research is as follows:  
\begin{thm}[Alon-Stav~\cite{AlonStavHereditaryStability}]
    For a nontrivial hereditary property $\hh$, there exists a $p^* = p_\hh^* \in [0,1]$ so that with $G\sim\GG(n,p^*)$, 
    \begin{align*}
        \lim_{n\to\infty} \max_{|V(G)| = n}\dist(G,\hh)
        = \Exp\left[\dist(G,\hh)\right]+o(1).  
    \end{align*}
\end{thm}~\\

In other words, random graphs of density $p^*$ asymptotically achieve the maximum distance to $\hh$.  
For any $p\in [0,1]$ and any property $\hh$ nontrivial and hereditary, let 
\begin{align}
    \ed_\hh(p) 
    := 
        \limsup_{n\to\infty}
            \max_{\substack{
                    |V(G)| = n, \\
                    e(G) = \lfloor p\binom{n}{2}\rfloor
            }} \dist(G,\hh) .
\end{align}~\\

We call $\ed_\hh$ the {\em edit distance function of $\hh$}.  Theorem~\ref{thm:distance to random graph} below demonstrates that the maximum distance $\hh$ among all density-$p$ graphs is achieved asymptotically by Erd\H{o}s-R\'enyi random graphs of expected density $p$. \\

\begin{thm}[Balogh-Martin~\cite{BaloghMartinEditComputation}]\label{thm:distance to random graph}
    Let $\hh$ be a nontrivial hereditary property.  
    For all $p\in [0,1]$, if $G\sim\GG(n,p)$, then 
    \begin{align*}
        \ed_\hh(p) = \lim_{n\to\infty} \Exp[\dist(G,\hh)].  
    \end{align*}
    Moreover the function $\ed_\hh$ is continuous and concave-down.  
\end{thm}~\\

Proposition~\ref{prop:ed of G(n0,1/2)} below has several short proofs and follows from Bollob\'as' asymptotic result on the chromatic number of a random graph (see~\cite{BollobasChromatic}), together with established techniques for computing edit distance functions (see \cite{MartinSurvey}). \\

\begin{prop}[Alon-Stav~\cite{AlonStavHereditaryStability}]\label{prop:ed of G(n0,1/2)}
    Let $F\sim\GG(n_0,1/2)$ and define $\hh:=\forb(F)$.  
    Then \aas~with $n_0\to\infty$, 
    \begin{align*}
        \ed_\hh(p) 
        = (1+o(1))\, 
            \dfrac{2\log_2n_0}
            {n_0}
            \cdot\min\{p,1-p\}.  
    \end{align*}
\end{prop}~\\

Our main result extends Proposition \ref{prop:ed of G(n0,1/2)} so that we are able to determine the edit distance function asymptotically for all $p_0$ in a relatively large open interval around $1/2$. Let $\varphi = (1+\sqrt{5})/2$ be the golden ratio. Note that $1-\varphi^{-1}\approx 0.381966$ and $\varphi^{-1}\approx 0.618034$. \\

\begin{thm}\label{thm:ed of random graph}
    Fix $p_0 \in (0,1)$, let $F\sim\GG(n_0,p_0)$, and define $\hh:=\forb(F)$.
    If $p_0\in [\phirange]$ then \aas~with $n_0\rightarrow\infty$,
    \begin{align}\label{eq:ed of G(n0,p0)}
        \ed_\hh(p) = (1+o(1))\,\dfrac{2\log n_0}{n_0}\cdot\min\left\{
            \dfrac{p}{-\log(1-p_0)}, 
            \dfrac{1-p}{-\log p_0}
        \right\} 
    \end{align}
    holds for all $p\in [0,1]$. If $p_0\in [0,1-\varphi^{-1})$, then \aas~\eqref{eq:ed of G(n0,p0)} holds for all $p\in [1/3,1]$. If $p_0\in (\varphi^{-1},1]$, then \aas~\eqref{eq:ed of G(n0,p0)} holds for all $p\in [0,2/3]$.
\end{thm}~\\
In fact, the $o(1)$ error term depends only on the constant $p_0$ and holds uniformly for all $p$ in each of the respective intervals.

The first author conjectured (see~\cite{MartinSurvey}) that for all $p_0\in [0,1]$, \eqref{eq:ed of G(n0,p0)} holds \aas~for all $p\in [0,1]$. Theorem~\ref{thm:ed of random graph} proves this for a range of $p_0$ of size $\approx 0.236068$. \\

\subsection{Equivalent parameters}\label{sub-sec:related parameters}
The edit distance function is also interesting because of its connection to other parameters involving random graphs.  
For $\hh$ any nontrivial hereditary property and any $p\in (0,1)$, the {\em speed} of $\hh$ is 
\begin{align*}
    c_\hh(p) := \lim_{k\to\infty} -\log_2\left(\PP\left[\GG(k,p) \in \hh\right]\right)\cdot \textstyle{\binom{n}{2}}^{-1}
\end{align*}
Indeed, this limit does exist and a proof of that fact appears in~\cite{AlekseevHereditary} and in~\cite{BollobasThomasonStructureColoring}.  
See also the survey~\cite{MartinSurvey}. \\

The following observation was made by Thomason but it can be shown to follow from a prior result due to Bollob\'as and Thomason~\cite{BollobasThomasonStructureColoring}. \\

\begin{thm}[Thomason~\cite{ThomasonGraphsColors}]\label{thm:c and ed}
    Let $\hh$ be a nontrivial hereditary property.  
    Then for all $p\in (0,1)$, 
    \begin{align*}
        c_\hh(p)    =   (-\log_2(p(1-p))) \cdot \ed_\hh\left(\dfrac{\log (1-p)}{\log(p(1-p))}\right).  
    \end{align*}
\end{thm}~\\

Note that the function $f:(0,1)\to(0,1)$ defined by 
\begin{align*}
    f(x) := \dfrac{\log (1-x)}{\log(x(1-x))}
\end{align*}
on $x\in (0,1)$ is invertible.  
Since $\ed_\hh$ is continuous, $c_\hh$ can be computed from $\ed_\hh$ and vice versa. As a result, combining Theorem~\ref{thm:c and ed} with Theorem~\ref{thm:ed of random graph} yields a result on the speed function of hereditary properties defined by random graphs. \\

\begin{cor}\label{cor:ed of random graph}
    Fix $p_0 \in (0,1)$, let $F\sim\GG(n_0,p_0)$, and define $\hh:=\forb(F)$.
    If $p_0\in [\phirange]$ then \aas~ with $n_0\to\infty$, 
    \begin{align}\label{eq:speed of G(n0,p0)}
        c_\hh(p)    =   (1+o(1))\,\dfrac{2\log_2 n_0}{n_0}\cdot\min\left\{
            \dfrac{\log (1-p)}{\log(1-p_0)}, 
            \dfrac{\log p}{\log p_0}
        \right\}
   \end{align}
    holds for all $p\in [0,1]$. If $p_0\in [0,1-\varphi^{-1})$, then \aas~\eqref{eq:speed of G(n0,p0)} holds for all $p\in [1-\varphi^{-1},1]$. If $p_0\in (\varphi^{-1},1]$, then \aas~\eqref{eq:speed of G(n0,p0)} holds for all $p\in [0,\varphi^{-1}]$.

\end{cor}~\\

For any hereditary property $\hh$ and any graph $G$, let $\chi_\hh(G)$ be the {\em $\hh$-chromatic number} of $G$. This is the minimum nonnegative integer $k$ for which there exists a partition $V(G) = V_1\udot\cdots\udot V_k$ such that $G[V_i]$ satisfies $\hh$ for all $i\in\{1,\ldots,k\}$.  
If $\hh$ is the property of being an empty graph, then $\chi_\hh(G)$ is the chromatic number of $G$.  \\ 

Bollob\'as and Thomason established Theorem~\ref{thm:H-chromatic number} for the $\hh$-chromatic number of a random graph. 
\begin{thm}[Bollob\'as-Thomason~\cite{BollobasThomasonStructureColoring}]\label{thm:H-chromatic number}
    Let $p\in (0,1)$ and let $\hh$ be a nontrivial hereditary property.  
    Then \aas~with $G\sim\GG(n,p)$, 
    \begin{align}\label{eq:H-chromatic number}
        \chi_{\hh}(G) = (1+o(1))\, c_\hh(p)\, \dfrac{n}{2\log_2n}
    \end{align}
\end{thm}~\\

Bollob\'as' classic asymptotic result~\cite{BollobasChromatic} on the chromatic number of the random graph can be derived from Theorem~\ref{thm:H-chromatic number} by observing that if $\hh_{\rm em}$ is the property of being an empty graph, then $c_{\hh_{\rm em}}(p) = -\log_2 (1-p)$ and so $\chi_{\hh_{\rm em}}(G)=(1+o(1))\, \frac{n}{2 \log_{1/(1-p)} n}$~\aas \\

However, the fact that $c_{\hh_{\rm em}}(p)=-\log_2 (1-p)$ can itself be derived from Theorem~\ref{thm:c and ed} and the entirely trivial observation that $\ed_{\hh_{\rm em}}(p)=p$. In general, $\chi_{\hh}$ has a close relationship with both $c_{\hh}$ and $\ed_{\hh}$. \\

The rest of the paper is organized as follows: In Section~\ref{sec:CRGs}, we discuss colored regularity graphs (CRGs) and prove some basic results that have the potential to apply to a wide variety of edit distance results beyond the scope of this paper. In Section~\ref{sec:main result}, we give the proof of Theorem~\ref{thm:ed of random graph}. Section~\ref{sec:discussion} includes a proof of the fact that for all $p\in[\phirange]$, $\ed_{\hh}(p)$ can be computed by a set CRGs whose order is bounded by a constant depending only on $\hh$. Section~\ref{sec:discussion} also includes a discussion of the role paths play in CRGs.
In Section~\ref{sec:questions}, we discuss open questions and potential future work. \\
\section{Colored regularity graphs}
\label{sec:CRGs}
In this section, we address colored regularity graphs. In Section~\ref{sub-sec:Background on CRGs}, we address background and basic facts about colored regularity graphs. Section~\ref{sub-sec:p-prohibited CRGs} discusses the new notion of $p$-prohibited CRGs. Lemma~\ref{lem:eigenvalue prohibited} and Lemma~\ref{lem:dalmatian p-core} are important new results on $p$-prohibited CRGs. They are proven in Section~\ref{sub-sec:eigenvalue prohibited} and Section~\ref{sub-sec:dalmatian p-core}, respectively.

\subsection{Background on CRGs}
\label{sub-sec:Background on CRGs}

The key element to studying the edit distance problem is the colored regularity graph, which was defined by Alon and Stav~\cite{AlonStavHereditaryStability} but appeared as {\em types} in the prior literature by Bollob\'as and Thomason (see~\cite{BollobasThomasonStructureColoring}).
\begin{defn}
    A \textbf{colored regularity graph} $K$ is a complete graph, together with a partition $\vk = \vw(K)\udot \vb(K)$ of the vertex set into white and black vertices, and a partition $E(K) = \ew(K)\udot\eb(K)\udot\eg(K)$ of the edge set into white, black, and gray edges.  
    
    A CRG $K'$ is called a \textbf{sub-CRG} of CRG $K$ (denoted $K'\subseteq K$) if $K'$ is obtained by deleting some vertices from $K$ and all incident edges. 
\end{defn}~\\

CRGs approximate large graphs and we want to know whether a forbidden graph $F$ is in a graph approximated by a given CRG. We express this in terms of embeddings of graphs into CRGs. \\

\begin{defn}
    A graph $F$ \textbf{embeds into} CRG $K$ (written $F\mapsto K$) if there exists a function $\phi:V(F)\to V(K)$ such that: 
    \begin{itemize}
    	\item If $uv\in E(F)$, then either $\phi(u)=\phi(v)\in\vbk$, or $\phi(u)\phi(v)\in\ebk\cup \egk$ 
    	\item If $uv\in E(F^c)$, then either $\phi(u)=\phi(v)\in\vwk$, or $\phi(u)\phi(v)\in\ewk\cup \egk$.  
    \end{itemize}
\end{defn}~\\

For any CRG $K$, we will treat the elements of $\RR^{V(K)}$ both as functions on $\vk$ and as vectors indexed by the vertices of $K$.  
For any two such $\x,\y\in \RR^{V(K)}$, we define $\langle \x,\y\rangle := \sum_{u\in\vk}\x(u)\y(u)$.   
We also let $M_K(p)\in\RR^{V(K)\times V(K)}$ be the matrix whose $uv$-th entry is 
\begin{equation}
	m_{uv} :=
	\left\{\begin{array}{rl}
		p, &\mbox{$u\neq v$ and $uv\in \ewk$, or $u=v$ and $u\in \vwk$;} \\
		1-p, &\mbox{$u\neq v$ and $uv\in \ebk$, or $u=v$ and $u\in \vbk$;} \\
		0, &\mbox{$u\neq v$ and $uv\in\egk$.}
	\end{array}
	\right.
\end{equation}~\\

The all-ones vector $\one\in\RR^{V(K)}$ is defined by $\one(u) = 1$ for all $u\in \vk$ and the all-zeroes vector is just $\zero = 0\cdot \one$.  
Furthermore, we let $\simplex_K$ be the {\em standard simplex associated to $K$} which consists of all $\x\in\RR^{V(K)}$ so that $\x\geq\zero$ in the component-wise sense and $\langle \x, \one\rangle = 1$.  
The elements of $\simplex_K$ will be called {\em weight vectors}. \\

Now define 
\begin{align*}
	g_K(p,\x) &:= \langle \x,M_K(p)\x\rangle\text{ and }\\
	g_K(p) &:= \min\left\{g_K(p,\x) : \x\in\simplex_K\right\}.  
\end{align*}
A weight vector $\x\in\simplex_K$ is said to be {\em optimal for $K$} if $g_K(p,\x) = g_K(p)$.  
For any $p\in [0,1]$, a CRG $K$ is said to be {\em $p$-core} if for any optimal weight vector $\x$, $\x(u) > 0$ for all $u\in V(K)$.  
It follows that for $K$ a $p$-core CRG, there exists a unique optimal weight vector. \\

For any hereditary property $\hh = \forb(\ff)$ we define the following family of CRGs: 
\begin{align*}
	\kk_\hh := \{K\text{ a CRG} : F\not\mapsto K\text{ for all }F\in\ff\}.  
\end{align*}~\\

Theorem \ref{thm:inf = min} is the main technique for computing $\ed_\hh(p)$, hence understanding the set $\kk_\hh$ is crucial to understanding $\ed_\hh(p)$. The first equality was given by Balogh and Martin~\cite{BaloghMartinEditComputation} and the second by Marchant and Thomason~\cite{MarchantThomasonExtremalColors}. \\

\begin{thm}\label{thm:inf = min}
	Let $\hh$ be a nontrivial hereditary property.  
	Then for all $p\in [0,1]$, 
	\begin{equation}
		\ed_\hh(p) = \inf_{K\in\kk_\hh}g_K(p) = \min_{K\in\kk_\hh} g_K(p).  
	\end{equation}
\end{thm}~\\

It follows by definition that the minimum in Theorem~\ref{thm:inf = min} is obtained by a $p$-core CRG and as Theorem~\ref{thm:p-core colors} shows, $p$-core CRGs have a well-defined structure. \\

\begin{thm}[Marchant-Thomason~\cite{MarchantThomasonExtremalColors}]\label{thm:p-core colors}
    Let $p\in [0,1]$ and suppose $K$ is a $p$-core CRG.  
    \begin{enumerate}[(a)]
        \item If $p\in [0,1/2]$, then $\ebk = \emptyset$ and for all $uv\in \ewk$, $u,v\in\vbk$.  
        \item If $p\in [1/2,1]$, then $\ewk = \emptyset$ and for all $uv\in \ebk$, $u,v\in \vwk$.  
    \end{enumerate}
\end{thm}~\\

To summarize, if $p\leq 1/2$, then a $p$-core CRG has no black edges and all white edges must be between black vertices. If $p\geq 1/2$, then a $p$-core CRG has no white edges and all black edges must be between white vertices. As a result, if $p=1/2$, $p$-core CRGs have neither black nor white edges. \\

\begin{rem}\label{rem:1/2 core}
    A CRG $K$ is $1/2$-core if and only if all edges of $K$ are gray.  
\end{rem}~\\

\subsection{$p$-prohibited CRGs}
\label{sub-sec:p-prohibited CRGs}

In this paper, we introduce the notion of a prohibited CRG.\\ 
\begin{defn}
    For any $p\in [0,1]$ and any CRG $J$, we say that $J$ is \textbf{$p$-prohibited} if for any $p$-core CRG $K$, $J$ is not a sub-CRG of $K$.  
\end{defn}~\\

For example, Theorem~\ref{thm:p-core colors} shows that if $p\in [0,1/2)$, then the only $2$-vertex CRGs that are not $p$-prohibited are those with a gray edge or the CRG with two black vertices and a white edge. See Figure~\ref{fig:p-prohibited}. \\

\begin{rem}\label{rem:prohibited vs core}
    There is an abundance of CRGs which are neither $p$-core nor $p$-prohibited.  For example, consider the CRGs $K$ and $K'$ defined as follows.  
    Let $K$ consist of $3$ black vertices with $2$ white edges and $1$ gray edge.  
    For all $p\in [0,1]$, $K$ is not $p$-core.  
    Now let $K'$ be the CRG on $4$ black vertices whose white edges induce a $P_4$ and all other edges are gray.  
    Clearly $K'$ contains $K$.  
    It is an exercise to see that $K'$ is $p$-core for all $p\in [0,1-\varphi^{-1})$, so $K$ is not $p$-prohibited on this interval.  
\end{rem}~\\

\begin{figure}[ht]
    \centering
    \scalebox{.6}[.6]{
        \usetikzlibrary{backgrounds}

\begin{tikzpicture}

	\begin{scope}
		\node [circle, draw=black, fill=white, minimum size = 1cm] (v1) at (-2,2) {};
		\node [circle, draw=black, fill=white, minimum size = 1cm] (v2) at (0,2) {};	
		\scoped[on background layer]{
    		\draw [double distance=4mm,thin]  (v1.center) -- (v2.center);
    	}
	\end{scope}
	
	\begin{scope}[xshift=4cm]
		\node [circle, draw=black, fill=white, minimum size = 1cm] (v1) at (-2,2) {};
		\node [circle, fill=black, minimum size = 1cm] (v2) at (0,2) {};	
		\scoped[on background layer]{
    		\draw [double distance=4mm,thin]  (v1.center) -- (v2.center);
    	}
	\end{scope}	
	
	\begin{scope}[yshift=-2cm]
		\node [circle, draw=white, fill=black, minimum size = 1cm] (v1) at (-2,2) {};
		\node [circle, draw=white, fill=black, minimum size = 1cm] (v2) at (0,2) {};	
		\scoped[on background layer]{
    		\draw [line width = .35cm,draw=black]  (v1.center) -- (v2.center);
    	}
	\end{scope}		 
	
	\begin{scope}[yshift=-2cm, xshift=4cm]
		\node [circle, draw=white, fill=black, minimum size = 1cm] (v1) at (-2,2) {};
		\node [circle, draw=black, fill=white, minimum size = 1cm] (v2) at (0,2) {};	
		\scoped[on background layer]{
    		\draw [line width = .35cm,draw=black]  (v1.center) -- (v2.center);
    	}
	\end{scope}
	\begin{scope}[yshift=-4cm, xshift=2cm]
		\node [circle, draw=black, fill=white, minimum size = 1cm] (v1) at (-2,2) {};
		\node [circle, draw=black, fill=white, minimum size = 1cm] (v2) at (0,2) {};	
		\scoped[on background layer]{
    		\draw [line width = .35cm,draw=black]  (v1.center) -- (v2.center);
    	}
	\end{scope}		 
	
	\begin{scope}[yshift=-.5cm, xshift=11cm]
		\node [circle, draw=white, fill=black, minimum size = 1cm] (v1) at (-4,1.5) {};
		\node [circle, draw=white, fill=black, minimum size = 1cm] (v2) at (-2,1.5) {};	
		\scoped[on background layer]{
    		\draw [double distance=4mm,thin]  (v1.center) -- (v2.center);
    	}
	\end{scope}		 
	
	\begin{scope}[yshift=-1.5cm, xshift=9cm]
		\node [circle, draw=white, fill=black, minimum size = 1cm] (v1) at (-2,0.5) {};
		\node [circle, draw=white, fill=black, minimum size = 1cm] (v2) at (0,0.5) {};	
		\scoped[on background layer]{
    		\draw [line width = .35cm,draw=gray]  (v1.center) -- (v2.center);
    	}
	\end{scope}

	\begin{scope}[yshift=-2.5cm, xshift=13cm]
		\node [circle, draw=white, fill=black, minimum size = 1cm] (v1) at (-2,3.5) {};
		\node [circle, draw=black, fill=white, minimum size = 1cm] (v2) at (0,3.5) {};	
		\scoped[on background layer]{
    		\draw [line width = .35cm,draw=gray]  (v1.center) -- (v2.center);
    	}
	\end{scope}		 
	
	\begin{scope}[yshift=-4cm, xshift=11cm, shift={(0,0.5)}]
		\node [circle, draw=black, fill=white, minimum size = 1cm] (v1) at (0,2.5) {};
		\node [circle, draw=black, fill=white, minimum size = 1cm] (v2) at (2,2.5) {};	
		\scoped[on background layer]{
    		\draw [line width=0.35cm,draw=gray]  (v1.center) -- (v2.center);
    	}
	\end{scope}

\end{tikzpicture}
    }
    \caption{All two-vertex CRGs. The five on the left are $p$-prohibited for all $p\in [0,1/2]$. The four on the right are $p$-core for all $p\in [0,1/2)$.}
    \label{fig:p-prohibited}
\end{figure}
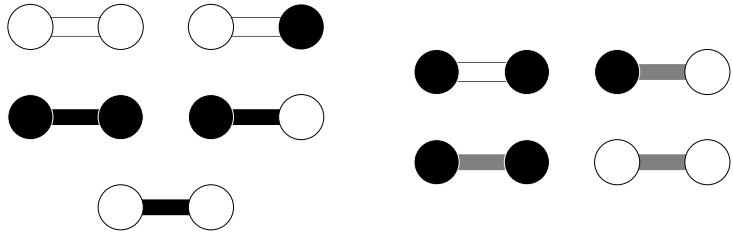~\\

We also want to introduce the notion of the complement of a CRG. \\

\begin{defn}
    If $K$ is a CRG, then the \textbf{complement} of $K$ is the unique CRG $\overline{K}$, such that 
    \begin{itemize}
        \item $\vw(\overline{K})=\vbk$, $\vb(\overline{K})=\vwk$, 
        \item $\ew(\overline{K})=\ebk$, $\eb(\overline{K})=\ewk$, and $\eg(\overline{K})=\egk$.
    \end{itemize}
\end{defn}~\\

For a graph $G$, the notation is $G^c$ is used to denote the graph complement, so as to avoid confusion. There is symmetry in the edit distance function about $p=1/2$ with respect to complements. \\

\begin{prop}\label{prop:complement g}
    If $p\in[0,1]$ and $K$ is a CRG, then $g_K(p)=g_{\overline{K}}(1-p)$. 
\end{prop}
\begin{proof}  
    This follows from the equality of the matrices $M_K(p)=M_{\overline{K}}(1-p)$:
    \begin{align*}
        g_K(p)  &=  \min\left\{\langle \x,M_K(p)\x\rangle : \x\in\simplex_K\right\} \\
                &=  \min\left\{\langle \x,M_{\overline{K}}(1-p)\x\rangle : \x\in\simplex_K\right\} = g_{\overline{K}}(1-p) . \qedhere
    \end{align*}
\end{proof}~\\

There is also symmetry in the edit distance function about $p=1/2$ when it comes to $p$-prohibition. \\

\begin{prop}\label{prop:Jbar prohibited}
    For all $p\in[0,1]$, a CRG $J$ is $p$-prohibited if and only if $\overline{J}$ is $(1-p)$-prohibited. 
\end{prop}
\begin{proof}
    Suppose $J$ is $p$-prohibited but $\overline{J}$ is not $(1-p)$-prohibited. Then, there is a $(1-p)$-core CRG $\overline{K}$ that contains $\overline{J}$ as a sub-CRG. If $K$ is not $p$-core, then there is a $K'\subseteq K$ such that $g_{K'}(p)=g_K(p)$, but Proposition~\ref{prop:complement g} gives that $g_{\overline{K'}}(1-p)=g_{\overline{K}}(1-p)$, a contradiction to $\overline{K}$ being $(1-p)$-core.
\end{proof}~\\

Next, we introduce terminology which is useful in describing the structure of $p$-core and $p$-prohibited CRGs. \\ 

\begin{defn}\label{defn:underlying}
    Let $K$ be a CRG.
    \begin{itemize}
        \item The \textbf{underlying graph} of $K$ is the graph $G = (\vk, \linebreak[1] \ebk \cup \ewk)$.  
        \item A \textbf{component} of $K$ is a component of the underlying graph of $K$.
        \item A \textbf{disjoint union} of vertex-disjoint CRGs $J,K$, denoted $J\oplus K$, is a CRG with vertex set $V_J\oplus V_K$, where the sub-CRG induced on $V_J$ is isomorphic to $J$, the sub-CRG induced on $V_K$ is isomorphic to $K$, and every edge incident to a vertex in each of $V_J$ and in $V_K$ has color gray. The disjoint union of $k$ copies of $K$ is $k\cdot K$. 
        \item Let $G$ be a nonempty graph. The CRG, $K$, \textbf{associated} to $G$ is defined as follows: If $p\in [0,1/2]$, then $\vwk=\ebk=\emptyset$, $\vbk=V(G)$, $\ewk=E(G)$, and $\egk=E(\overline{G})$. If $p\in (1/2,1]$, then $\vbk=\ewk=\emptyset$, $\vwk=V(G)$, $\ebk=E(G)$, and $\egk=E(\overline{G})$. 
    \end{itemize}
\end{defn}~\\

We associate CRGs to graphs for the purposes of discussing $p$-core CRGs.  
See Figure~\ref{fig:component example} for an example.  
Since $1/2$-core CRGs are precisely those which have only gray edges, the definition of the CRG associated to a graph for $p=1/2$ is made purely out of convenience.  

\begin{figure}[ht]
    \centering
        \scalebox{.6}[.6]{
            \usetikzlibrary{backgrounds}

\begin{tikzpicture}
	\pgfdeclarelayer{background}
	\pgfdeclarelayer{foreground}
	\pgfsetlayers{background,main,foreground}
	
	\begin{pgfonlayer}{foreground}
	    \node [circle, draw=white, fill=black, minimum size = 1cm] (b1) at (90:4) {};
	    \node [circle, draw=white, fill=black, minimum size = 1cm] (b2) at (50:4) {};
	    \node [circle, draw=white, fill=black, minimum size = 1cm] (b3) at (10:4) {};
	    \node [circle, draw=white, fill=black, minimum size = 1cm] (b4) at (-30:4) {};

	    \node [circle, draw=white, fill=black, minimum size = 1cm] (b5) at (-70:4) {};
	    \node [circle, draw=white, fill=black, minimum size = 1cm] (b6) at (-110:4) {};

	    \node [circle, draw=black, fill=white, minimum size = 1cm] (w1) at (130:4) {};
	    \node [circle, draw=black, fill=white, minimum size = 1cm] (w2) at (170:4) {};
	    \node [circle, draw=black, fill=white, minimum size = 1cm] (w3) at (210:4) {};
	\end{pgfonlayer}
	\begin{pgfonlayer}{main}
    	\draw [double distance=5mm,thick]  (b1.center) -- (b2.center);
    	\draw [double distance=5mm,thick]  (b2.center) -- (b3.center);
    	\draw [double distance=5mm,thick]  (b3.center) -- (b4.center);
    	\draw [double distance=5mm,thick]  (b1.center) -- (b4.center);
    	
		\draw [double distance=5mm,thick] (b5.center) -- (b6.center);
 	\end{pgfonlayer}
 	
 	\begin{pgfonlayer}{background}
		\draw [line width = .25cm,draw=gray] (b1.center) -- (b3.center);
		\draw [line width = .25cm,draw=gray] (b1.center) -- (b4.center);
		\draw [line width = .25cm,draw=gray] (b2.center) -- (b4.center);

		\draw [line width = .25cm,draw=gray] (w1.center) -- (w2.center);
		\draw [line width = .25cm,draw=gray] (w1.center) -- (w3.center);
		\draw [line width = .25cm,draw=gray] (w2.center) -- (w3.center);

		\draw [line width = .25cm,draw=gray] (w1.center) -- (b1.center);
		\draw [line width = .25cm,draw=gray] (w1.center) -- (b2.center);
		\draw [line width = .25cm,draw=gray] (w1.center) -- (b3.center);
		\draw [line width = .25cm,draw=gray] (w1.center) -- (b4.center);
		\draw [line width = .25cm,draw=gray] (w1.center) -- (b5.center);
		\draw [line width = .25cm,draw=gray] (w1.center) -- (b6.center);

		\draw [line width = .25cm,draw=gray] (w2.center) -- (b1.center);
		\draw [line width = .25cm,draw=gray] (w2.center) -- (b2.center);
		\draw [line width = .25cm,draw=gray] (w2.center) -- (b3.center);
		\draw [line width = .25cm,draw=gray] (w2.center) -- (b4.center);
		\draw [line width = .25cm,draw=gray] (w2.center) -- (b5.center);
		\draw [line width = .25cm,draw=gray] (w2.center) -- (b6.center);

		\draw [line width = .25cm,draw=gray] (w3.center) -- (b1.center);
		\draw [line width = .25cm,draw=gray] (w3.center) -- (b2.center);
		\draw [line width = .25cm,draw=gray] (w3.center) -- (b3.center);
		\draw [line width = .25cm,draw=gray] (w3.center) -- (b4.center);
		\draw [line width = .25cm,draw=gray] (w3.center) -- (b5.center);
		\draw [line width = .25cm,draw=gray] (w3.center) -- (b6.center);

		\draw [line width = .25cm,draw=gray] (b5.center) -- (b1.center);
		\draw [line width = .25cm,draw=gray] (b5.center) -- (b2.center);
		\draw [line width = .25cm,draw=gray] (b5.center) -- (b3.center);
		\draw [line width = .25cm,draw=gray] (b5.center) -- (b4.center);

		\draw [line width = .25cm,draw=gray] (b6.center) -- (b1.center);
		\draw [line width = .25cm,draw=gray] (b6.center) -- (b2.center);
		\draw [line width = .25cm,draw=gray] (b6.center) -- (b3.center);
		\draw [line width = .25cm,draw=gray] (b6.center) -- (b4.center);
	\end{pgfonlayer}
\end{tikzpicture}
        }
    \caption{
        A CRG with $5$ components.  
        One component is the CRG associated to the cycle $C_4$ for $p\in [0,1/2)$.  
        The edges satisfy the necessary conditions from Theorem \ref{thm:p-core colors} for a $p$-core CRG with $p\in [0,1/2)$.  
    }
    \label{fig:component example}
\end{figure}
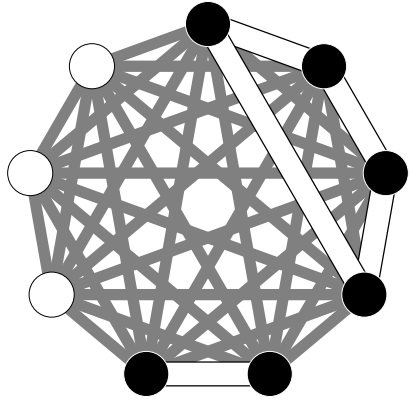~\\

In order to apply Lemma~\ref{lem:eigenvalue prohibited} below, we need the minimum adjacency eigenvalue to be at most $-1$. This occurs for all nonempty graphs. See~\cite{SpectraGraphBook} for a more detailed discussion about eigenvalues associated to graphs.  
\\

\begin{prop}\label{prop:small eigenvalue}
    Every nonempty graph that is not disjoint cliques has minimum adjacency eigenvalue at most $-\sqrt{2}$. 
    If a nonempty graph consists of disjoint cliques, its minimum adjacency eigenvalue is $-1$. 
    An empty graph has all adjacency eigenvalues zero.
\end{prop}~\\
%

\begin{lem}\label{lem:eigenvalue prohibited}
	Let $G$ be a nonempty graph and let $\lambda\leq-1$ be the minimum eigenvalue of the adjacency matrix of $G$. The CRG associated to $G$ is $p$-prohibited for all
	\begin{align*}
	    p\in \left[\dfrac{1}{1-\lambda},1-\dfrac{1}{1-\lambda}\right] .  
	\end{align*}
\end{lem}~\\

In Section~\ref{sub-sec:eigenvalue prohibited}, we prove Lemma~\ref{lem:eigenvalue prohibited}. 
First, we need some essential terms. \\

\begin{defn}
    For any positive integer $t$, the \textbf{$t$-dalmatian CRG} is the CRG, denoted $D_t$, consisting of $t$ black vertices and all edges white.  The \textbf{$\infty$-dalmatian CRG} is the CRG, denoted $D_\infty$, which is a single white vertex. 
    The set of CRGs denoted by $\dd_p$ is as follows:
    \begin{itemize}
        \item   If $p\in [0,1/2)$, then $\dd_p$ is the set of all CRGs whose components are dalmatian CRGs.
        \item   If $p\in (1/2,1]$, then $\dd_p$ is the set of all CRGs whose components are complements of dalmatian CRGs.
        \item   If $p=1/2$, then $\dd_{1/2}$ is the set of all CRGs whose components are single vertices.
    \end{itemize}
\end{defn}~\\

See Figure~\ref{fig:dalmatian} for dalmatian CRGs of small order.

\begin{figure}[ht]
    \centering
    \scalebox{.6}[.6]{
        \begin{tikzpicture}
	\begin{scope}
		\node [circle, draw=black, minimum size = 1cm] at (-2,2) {};
	\end{scope}
	
	\begin{scope}[xshift=2cm]
		\node [circle, fill=black, minimum size = 1cm] at (-2,2) {};
	\end{scope}
	
	\begin{scope}[xshift=4cm]
		\node [circle, fill=black, minimum size = 1cm] (v1) at (-2,2) {};
		\node [circle, fill=black, minimum size = 1cm] (v2) at (0,2) {};
		\scoped[on background layer]{
    		\draw [double distance=4mm,thin]  (v1.center) -- (v2.center);
    	}
	\end{scope}	
	
	\begin{scope}[xshift=8cm, yshift=1cm]
		\node [circle, fill=black, minimum size = 1cm] (v1) at (-2,2) {};
		\node [circle, fill=black, minimum size = 1cm] (v2) at (0,2) {};
		\node [circle, fill=black, minimum size = 1cm] (v3) at (-1,0) {};	
		\scoped[on background layer]{
    		\draw [double distance=4mm,thin]  (v1.center) -- (v2.center);
    		\draw [double distance=4mm,thin]  (v1.center) -- (v3.center);
    		\draw [double distance=4mm,thin]  (v2.center) -- (v3.center);
    	}
	\end{scope}		 
	
	\begin{scope}[yshift=1cm, xshift=12cm]
		\node [circle, fill=black, minimum size = 1cm] (v1) at (-2,2) {};
		\node [circle, fill=black, minimum size = 1cm] (v2) at (0,2) {};
		\node [circle, fill=black, minimum size = 1cm] (v3) at (-2,0) {};	
		\node [circle, fill=black, minimum size = 1cm] (v4) at (0,0) {};
		\scoped[on background layer]{
    		\draw [double distance=4mm,thin]  (v1.center) -- (v3.center);
    		\draw [double distance=4mm,thin]  (v2.center) -- (v4.center);
    		
    		\draw [double distance=4mm,thin]  (v1.center) -- (v2.center);
    		\draw [double distance=4mm,thin]  (v2.center) -- (v3.center);
    		\draw [double distance=4mm,thin]  (v3.center) -- (v4.center);
    		\draw [double distance=4mm,thin]  (v4.center) -- (v1.center);
    	}
	\end{scope}	
\end{tikzpicture}
    }
    \caption{The dalmatian CRGs $D_\infty=\overline{D_1},\, D_1=\overline{D_{\infty}},D_2, D_3,$ and $D_4$.}
    \label{fig:dalmatian}
\end{figure}
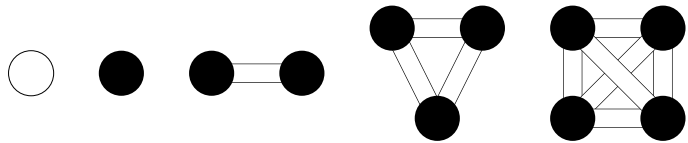~\\

\begin{rem}\label{rem:dalmatian p-core}
    For all $p\in [0,1]$ and each $K\in \dd_p$, $K$ is a $p$-core CRG. 
    Moreover, if $p\in [0,1/2]$, then $g_{D_\infty}(p)=p$ and for each positive integer $t$,  
    \begin{align*}
        g_{D_t}(p) = \min_{\x\in\simplex_{D_t}}\langle \x, M_{D_t}(p)\x\rangle 
        = \dfrac1{t^2}\langle \one, M_{D_t}(p)\one\rangle 
        = p + \dfrac{1-2p}{t} .
    \end{align*}
    If $p\in [1/2,1]$, then $g_{\overline{D_\infty}}(p)=1-p$ and for each positive integer $t$,  
    $g_{\overline{D_t}}(p) = 1-p + \frac{2p-1}{t}$.
\end{rem}~\\

In Lemma~\ref{lem:dalmatian p-core}, we show that for all $p\in [\phirange]$, the {\em only} $p$-core CRGs are those that belong to $\dd_p$. \\

\begin{lem}\label{lem:dalmatian p-core}
	For $p\in [0,1]$, the CRG associated to $P_3$ is $p$-prohibited if and only if $p\in [\phirange]$.  
    In particular for all $p\in [\phirange]$, a CRG $K$ is $p$-core if and only if $K\in\dd_p$.
\end{lem}~\\

\begin{rem}
    Since the minimum eigenvalue of the adjacency matrix of $P_3$ is $-\sqrt{2}$, Lemma~\ref{lem:eigenvalue prohibited} gives that $P_3$ is prohibited for $p$ in the interval $[\sqrt{2}-1,2-\sqrt{2}]\approx [0.414,0.586]$. Lemma~\ref{lem:dalmatian p-core}, however, gives that $P_3$ is prohibited on the larger interval $[\phirange]\approx [0.382,0.618]$.
\end{rem}~\\

In Section \ref{sub-sec:dalmatian p-core}, we prove Lemma~\ref{lem:dalmatian p-core}. \\

\subsection{Proof of Lemma~\ref{lem:eigenvalue prohibited}}
\label{sub-sec:eigenvalue prohibited}

The result from Theorem~\ref{thm:degree} below was originally shown by Sidorenko~\cite{SidorenkoOptimalMultigraph} using different language and has appeared in several other forms throughout the study of hereditary properties.  
See \cite{MartinSurvey} for a more detailed history.  
For convenience, we state it in the language of CRGs. \\

\begin{thm}\label{thm:degree}
	Let $K$ be a $p$-core CRG with optimal weight vector $\x\in\simplex_K$.  
	Then 
	\begin{align*}\label{eq:M_K(p) formula}
		M_K(p)\,\x = g\,\one.  
	\end{align*}
\end{thm}~\\

So, Theorem~\ref{thm:degree} establishes that the optimal weight vector produces a weighting that is balanced. 
\begin{lem}\label{lem:delta prohibited}
    Let $0\leq p\leq 1$ and $J$ be a CRG.  
    If there exists a nonzero vector $\bdelta\in\RR^{V(J)}$ so that $\langle\bdelta,\one\rangle=0$ and 
    \begin{equation}
        \langle\bdelta, M_J(p)\,\bdelta\rangle \leq 0 ,  
    \end{equation}
    then $J$ is $p$-prohibited.  
\end{lem}
\begin{proof}
    We proceed by contradiction.  
    Suppose $J$ is not $p$-prohibited and that there exists $\bdelta$ as above.  
    Then there exists a $p$-core CRG $K$ containing $J$ and we may let $\x$ denote the optimal weight vector for $K$.  
    We extend $\bdelta$ to a vector $\bdelta'\in\RR^{V(K)}$ by letting $\bdelta'(u):=0$ if $u\in V(K)\setminus V(J)$ and $\bdelta'(u):=\bdelta(u)$ if $u\in V(J)$.  
    Note that $\langle\bdelta',\one\rangle=0$. \\
    
    Since $\x$ is the optimal weight vector for the $p$-core CRG $K$, $\x(u) > 0$ for all $u\in V(K)$ and it follows that there exists some $\varepsilon > 0$ so that $\x':=\x + \varepsilon\bdelta'$ lies in $\simplex_K$ and $\x'(u) = 0$ for some $u\in \vk$. 
    By the definition of $g_K(p)$, the fact that $\x$ is optimal, and Theorem \ref{thm:degree}, 
    \begin{align*}
        0 &< g_K(p,\x') - g_K(p,\x)\\
        &= \langle \x+\varepsilon\bdelta', M_K(p)(\x+\varepsilon\bdelta')\rangle 
        - \langle \x, M_K(p)\x\rangle\\
        &= 2\varepsilon\langle \bdelta, M_K(p)\x\rangle
        + \varepsilon^2\langle \bdelta', M_K(p)\bdelta'\rangle\\
        &= 2\varepsilon\langle \bdelta', g_K(p)\one\rangle
        + \varepsilon^2\langle \bdelta, M_J(p)\bdelta\rangle\\
        &= \varepsilon^2\langle \bdelta, M_J(p)\bdelta\rangle \leq 0, 
    \end{align*}
    a contradiction to the assumption that $J$ is not $p$-prohibited.  
\end{proof}~\\

We include one more well-known fact about $p$-core CRGs. \\

\begin{prop}[See~\cite{MartinSurvey}]\label{prop:archapelago}
    Let $K_1,\ldots,K_{\ell}$ be CRGs and let $K=K_1\oplus\cdots\oplus K_{\ell}$.  
    Then for all $p\in [0,1]$, 
    \begin{equation}
        g_{K}(p)^{-1} = \sum_{i=1}^{\ell} g_{K_i}(p)^{-1} .  
    \end{equation}
    In particular, $K$ is $p$-core if and only if each of $K_1,\ldots,K_{\ell}$ are $p$-core.  
\end{prop}~\\

\begin{lem}\label{lem:prohibited extension}
    Let $p\in [0,1]$. 
    A CRG $J$ is $p$-prohibited if and only if for all $p$-core CRGs $K$ and all positive integers $k$, the CRG $(k\cdot J)\oplus K$ is $p$-prohibited.  
\end{lem}
\begin{proof}
    To prove the forward implication, if $J$ is $p$-prohibited, then no $p$-core CRG can contain $k\cdot J$ because it would contain $J$.
    To prove the reverse implication, if $J$ is not $p$-prohibited then there exists a $p$-core CRG, $L$, containing $J$.  
    Then $J\oplus K$ is contained in $L\oplus K$, which is $p$-core by Proposition~\ref{prop:archapelago}, as desired.  
\end{proof}~\\

With the primary tools of Lemmas~\ref{lem:delta prohibited} and~\ref{lem:prohibited extension} established, we now proceed to prove Lemma~\ref{lem:eigenvalue prohibited} itself. \\

Recall that $J$ is the CRG associated to a nonempty graph $G$.
If $p=1/2$, then by Theorem \ref{thm:p-core colors}, $J$ is $1/2$-prohibited if and only if $J$ has an edge that is not gray.  
Thus, any nonempty $G$ gives that $J$ is $1/2$-prohibited, settling the case where $p=1/2$. \\  
    
Now suppose $p\in (0,1/2)$.  
Write $A$ for the adjacency matrix of $G$ and suppose $A\x = \lambda \x$ for some unit vector $\x$ where $\lambda$ is the minimum eigenvalue of $A$.  
Then $J$ is the the CRG on $V(G)$ with all vertices black, and where edge $uv$ is white if $uv\in E(G)$, and $uv$ is gray if $uv\in E(G^c)$.  
So $M_J(p) = (1-p)I + pA$. \\ 
    
For convenience, we write $V_1$ and $V_2$ for the vertex sets in $2\cdot J$ corresponding to each copy of $J$.  
Let $\bdelta\in\RR^{V(2\cdot J)}$ be defined by 
\begin{align*}
    \bdelta(u)  :=  \left\{     \begin{array}{rl}
                                    \x(u),   &   u\in V_1; \\
                                    -\x(u),  &   u\in V_2 .
                                \end{array}\right.
\end{align*}~\\
    
By definition, $\langle \one, \bdelta\rangle = 0$.  
Moreover note that 
\begin{align*}
    \langle\bdelta, M_{2\cdot J}(p)\bdelta\rangle 
    &=  \langle \x, M_J(p)\x \rangle + \langle -\x, M_J(p)(-\x) \rangle \\
    &=  2 \langle \x, M_J(p)\x \rangle  \\
    &=  2 \langle \x, ((1-p)I+pA)\x \rangle \\
    &=  2(1-p) \langle \x, I\x \rangle + 2p \langle \x, A\x \rangle \\
    &=  2(1-p) \langle \x, \x \rangle + 2p \lambda \langle \x, \x \rangle \\
    &=  2\cdot(1-(1-\lambda)p) \cdot \langle \x, \x \rangle.  
\end{align*}~\\
    
If $p\geq 1/(1-\lambda)$, then $\langle\bdelta, M_{2\cdot J}(p)\bdelta\rangle\leq 0$. By Lemma~\ref{lem:delta prohibited}, the CRG $2\cdot J$ is $p$-prohibited and by Lemma~\ref{lem:prohibited extension}, $J$ itself is $p$-prohibited. This settles the case where $p\in (0,1/2)$.\\
    
Finally, for the case of $p\in (1/2,1)$, Proposition~\ref{prop:Jbar prohibited} gives that $J$ is $p$-prohibited if and only if $J$ is $(1-p)$-prohibited. This concludes the proof of Lemma~\ref{lem:eigenvalue prohibited}.\qed \\

\subsection{Proof of Lemma~\ref{lem:dalmatian p-core}}
\label{sub-sec:dalmatian p-core}

Lemma~\ref{lem:dominant vertex} below is a result in pure graph theory that is reminiscent of the theorem that categorizes $\{P_4,C_4,C_4^c\}$-free graphs as threshold graphs. 
A \textit{dominant vertex} in a graph is one for which every other vertex is its neighbor. \\

\begin{lem}\label{lem:dominant vertex}
    If $G$ is a connected $\{P_4,C_4\}$-free graph, then $G$ has a dominant vertex.  
\end{lem}

\begin{proof}
    Let $u$ be a vertex of $G$ which attains the maximum degree $\Delta=\Delta(G)$ and let $A:=N_G(u)$ and $B:=V(G)\setminus(A\cup\{u\})$.  
    If $B = \emptyset$, then $u$ is the desired vertex, so we assume otherwise.  
    Let $w\in B$.  
    Since $G$ avoids induced $P_4$-s, connectivity implies $G$ has diameter at most $2$. \\
    
    In particular, $\dist_G(u,w) = 2$, so there exists some vertex $v\in A$ so that $uvw$ is an induced path on $3$ vertices. 
    If $v'$ is any vertex in $A\setminus\{v\}$, then $v'uvw$ is a path on $4$ vertices.  
    Since $uw\notin E(G)$ and $G$ avoids both induced $P_4$-s and induced $C_4$-s, it follows that $vv'\in E(G)$.  
    So $v$ is adjacent to $\{u,w\}\cup (A\setminus\{v\})$ and has degree at least $\Delta+1$, a contradiction.  
\end{proof}~\\

Lemma~\ref{lem:dominant vertex} yields a very strong structural theorem on CRGs, Lemma~\ref{lem:dominant to dalmatian}. Recall the definition of the underlying graph of CRG, $K$, in Definition~\ref{defn:underlying}: the graph whose vertices are the vertices of $K$ and whose edges are the non-gray edges of $K$. \\

\begin{lem}\label{lem:dominant to dalmatian}
    Let $p\in [1/3,2/3]$. 
    If a $p$-core CRG has an underlying graph which is $\{P_4,C_4\}$-free, then every component of the underlying graph is a clique. 
    That is, every component of such a CRG must be a member of $\dd_p$.
\end{lem}
\begin{proof}
    If $p=1/2$, then as we saw in Remark~\ref{rem:1/2 core}, a $p$-core CRG has only gray edges and so the underlying graph is empty. \\
    
    Let $p\in [1/3,1/2)$. 
    Every trivial component of the underlying graph is simply a vertex in the CRG.
    Let $K$ be a nontrivial component of the CRG. 
    By Theorem~\ref{thm:p-core colors}, the vertices of $K$ must be black and by Lemma~\ref{lem:dominant vertex}, the underlying graph of $K$ has a dominant vertex $u$. \\

    Let $\x\in\simplex_K$ be the optimal weight vector for $K$ and define $g := g_K(p)$.  
    By Theorem \ref{thm:degree}, $M_K(p)\x = g\one$ and by inspecting the entry indexed by $u$, 
    \begin{align*}
        g   =   (1-p)\x(u) + p\cdot\sum_{u\neq v\in \vk}\x(v)
            =   p + (1-2p)\x(u) > p.  
    \end{align*}~\\
    
    For a contradiction, we now suppose $K$ is not a dalmatian CRG. Hence, there exists some gray edge $vw$, and the sub-CRG $K'$ on $v$ and $w$ is the disjoint union of two black vertices.  
    Since $K$ is $p$-core, 
    \begin{align*}
        g 
        < g_{K'}(p) 
        = \min_{\y\in \simplex_{K'}}(1-p)(\y(v)^2+\y(w)^2) 
        = \dfrac{1-p}{2}.  
    \end{align*}
    Altogether, $p<g<(1-p)/2$ which implies that $p < 1/3$, a contradiction. \\
    
    The case where $p\in (1/2,2/3]$ follows by symmetry.
\end{proof}~\\

We now prove Lemma~\ref{lem:dalmatian p-core} itself with the primary tool being Lemma~\ref{lem:dominant to dalmatian}. 
As mentioned in Remark~\ref{rem:prohibited vs core}, we leave it as an exercise to verify that the CRG associated with $P_4$ is $p$-core for $p\in (0,1-\varphi^{-1})\cup (\varphi^{-1},1)$. 
Hence, $P_3$ is not $p$-prohibited in this range, proving the forward implication. \\

For the reverse implication, let $p\in [\phirange]$. 
Since the minimum eigenvalues of the adjacency matrices of $C_4$ and $P_4$ are $-2$ and $-\varphi^{-1}$ respectively, Lemma~\ref{lem:eigenvalue prohibited} implies that the CRGs associated with $C_4$ and $P_4$ are $p$-prohibited for $p\in [\phirange]$. \\

Suppose $K$ is a $p$-core CRG.  
Since $p\in [1/3,2/3]$, it follows from Lemma~\ref{lem:dominant to dalmatian} that the components of the underlying graph of $K$ are cliques.  
No such graph contains an induced $P_3$ and so $P_3$ is $p$-prohibited for all $p\in [\phirange]$, as desired.  \\

For the second statement of the theorem, since $P_3$ is $p$-prohibited, the only underlying graphs of a $p$-core CRG can be disjoint cliques, which is exactly the condition of being in $\dd_p$. As observed in Remark~\ref{rem:dalmatian p-core}, all CRGs in $\dd_p$ are $p$-core. This concludes the proof of Lemma~\ref{lem:dalmatian p-core}.\qed \\

\section{Proof of the main result}
\label{sec:main result}
    To proceed with the proof of Theorem~\ref{thm:ed of random graph}, we need some preparation.  
    In Section~\ref{sub-sec:trimming p-core CRGs}, Lemma~\ref{lem:bounded components} shows that for all $p\in (1/3,2/3)$ and for any CRG $K$, there exists a sub-CRG $K'$ of $K$ so that $g_{K'}(p)$ is close to $g_K(p)$ and $K'$ has components whose order is bounded by a function of $p$ and a tolerance term $\varepsilon$.  \\

    For the remaining discussion, let $p_0\in (0,1)$ and define 
    \begin{align}
        p^*     :=  \dfrac{\log(1-p_0)}{\log(p_0(1-p_0))} . \label{eq:p star}
    \end{align}
    In Section \ref{sub-sec:forbidding a random graph}, we investigate when a random graph $F\sim\GG(n_0,p_0)$ embeds into a CRG (i.e., when $F\mapsto K$).  
    There we show that \aas, if a CRG, $K$, has bounded components in the above sense and the random graph does not map into $K$, then $g_K(p^*)$ has to be at least the desired value to within a small tolerance.
    Applying Lemma~\ref{lem:bounded components}, this is true even if the components of $K$ are not bounded.  \\
    
    Finally in Section \ref{sub-sec:ed of random graph}, we put together these ideas to prove our main result.  

\subsection{Trimming $p$-core CRGs}
\label{sub-sec:trimming p-core CRGs}

The main result in this subsection is Lemma~\ref{lem:bounded components}, which establishes that, for $p\in(1/3,2/3)$, a CRG has a sub-CRG with bounded component sizes and a negligible change in the value of the $g$-function.   
\begin{lem}\label{lem:bounded components}
    Fix $p\in (1/3,2/3)$ and $\varepsilon\in(0,1)$.  
    There exists a positive integer $B = B(p,\varepsilon)$ such that the following holds: 
    For all CRGs $K$, there exists a $p$-core sub-CRG $K'$ whose components have order at most $B$, and $g_{K'}(p) \leq (1+\varepsilon)g_K(p)$.  
\end{lem}

The first part of the proof is to remove vertices from a CRG $K$ one-by-one in a way which does not affect $g_K(p)$ too much.  
Once enough vertices are removed, we show that each remaining vertex is incident to a bounded number of non-gray edges.  
Finally, we use Lemma~\ref{lem:eigenvalue prohibited} to bound the diameter of $p$-core CRGs on the interval $p\in (1/3,2/3)$.  
The underlying graph has bounded degree and diameter, thus its connected components have bounded order. \\

The proof consists of a sequence of propositions:
\begin{prop}\label{prop:g of sub-CRG}
	Fix $p\in [0,1]$ and suppose $K$ is a $p$-core CRG with least two vertices.  
	If $\x\in\simplex_K$ is the optimal weight vector for $K$, i.e., $g = g_K(p) = g_K(p,\x)$, then for all $u\in V(K)$, 
	\begin{align*}
		g_{K\setminus\{u\}}(p)  \leq    g + \dfrac{\x(u)^2}{(1-\x(u))^2}.  
	\end{align*}
\end{prop}
\begin{proof}
    Let $K':=K\setminus\{u\}$.  
    Since $K$ is $p$-core with at least two vertices, $\x(u)<1$ and we may define $\x'\in\simplex_K$ by 
    \begin{align*}
	    \x'(v):=\left\{\begin{array}{rl}
		    0, &v=u\\
		    \dfrac{\x(v)}{1-\x(u)}, &\text{otherwise}
	    \end{array}\right.
    	.
    \end{align*}
    In other words, if $\be_u\in\RR^{\vk}$ is the indicator vector for the vertex $u$, then $(1-\x(u))\x' = \x - \x(u)\be_u$.  
    Recall that $M_K(p)$ denotes the weighted adjacency matrix of $K$. 
    By Theorem~\ref{thm:degree}, $M_K(p)\x = g{\bf 1}$ and so 
    \begin{align*}
	    (1-\x(u))^2\langle \x', M_K(p)\x'\rangle 
	    &=      \langle \x-\x(u){\bf e}_u, M_K(p)(     \x-\x(u){\bf e}_u )\rangle \\
	    &=      \langle \x,g{\bf 1}\rangle - 2\x(u)\langle {\bf e}_u, g{\bf 1}\rangle + \x(u)^2\langle {\bf e}_u, M_K(p){\bf e}_u\rangle\\
	    &\leq   g - 2g\,\x(u) + \x(u)^2.  
    \end{align*}
    By definition of $g_{K'}(p)$ and since $\x'(u) = 0$, 
    \begin{align*}
	    g_{K'}(p) 
	    &\leq g_K(p,\x') \\
	    &= \dfrac{\langle \x', M_K(p)\x'\rangle}{(1-\x(u))^2} \\
	    &= \dfrac{g - 2g\,\x(u)+\x(u)^2}{(1-\x(u))^2} \\
	    &= g + \dfrac{(1 - g)\x(u)^2}{(1-\x(u))^2}\\
	    &\leq g+\dfrac{\x(u)^2}{(1-\x(u))^2}, 
    \end{align*}
    which completes the proof.  
\end{proof}~\\

\begin{prop}\label{prop:sub-CRG}
	Fix $p\in [0,1]$ and $\varepsilon\in (0,1)$. If $K$ is a $p$-core CRG with $g = g_K(p)$, then there exists a $p$-core sub-CRG $K'$ of $K$ such that the following holds: 
	\begin{enumerate}
	    \item $|V(K')|\leq 4/(\varepsilon g)$, 
	    \item $g_{K'}(p) \leq (1+17\varepsilon) g$, and 
	    \item if $\x'\in\simplex_{K'}$ is the optimal weight vector for $K'$, then
	    \begin{align*}
	        \min_{u\in V(K')}\x'(u)\geq \varepsilon g.
        \end{align*}
	\end{enumerate}
\end{prop}
\begin{proof}
    Define a finite sequence of sub-CRGs 
    \begin{align*}
        K = K_0\supset K_1\supset\cdots
        \supset K_\ell
    \end{align*}
    as follows:  
    First let $K:=K_0$ and $g_0 := g_K(p)$.  
    For any $k\geq 0$ so that $|V(K_k)| \geq 2$, do the following: 
    \begin{enumerate}[(i)]
        \item Let $\x_k\in\simplex_{K_k}$ be the optimal weight vector for $K_k$, i.e., $g_{K_k}(p,\x_k) = g_{K_k}(p)$.  
        \item Let $u_k\in V(K_k)$ so that $\x_k(u_k) = \min\left\{\x_k(v) : v\in V(K_k)\right\}$.  
        \item Let $K_{k+1}$ be any $p$-core sub-CRG of $K_k\setminus\{u_k\}$.  
    \end{enumerate}
    Since each step removes at least one vertex, $\ell\leq|\vk|$.  
    For each $k\in\{0,\ldots,\ell\}$, denote $g_k:=g_{K_k}(p)$. \\
    
    Let $a\in\{0,\ldots,\ell\}$ be the minimum index $a$ so that $|V(K_a)| \leq 4/(\varepsilon g)$.   
    By definition of $a$ and by the fact that $\varepsilon,g < 1$, 
    \begin{align*}
        \x_k(u_k)   \leq    1/|V(K_k)| 
                    \leq    (\varepsilon g)/4 \leq 1/4
    \end{align*}
    for all $k\in\{0,\ldots,a-1\}$.  
    By Proposition \ref{prop:g of sub-CRG}, 
    \begin{align*}
	    g_{k+1}     \leq    g_k + \dfrac{\x_k(u_k)^2}{(1-\x_k(u_k))^2} 
	                \leq    g_k + \dfrac{\x_k(u_k)^2}{(3/4)^2} 
	                <       g_k + \dfrac{2}{|V(K_k)|^2}.  
    \end{align*}
    Because $|V(K_k)| \geq |V(K_{a-1})| + (a-1-k)$ for all $k\in\{0,\ldots,a-1\}$,
    \begin{align*}
	    g_a     &<      g + \sum_{ k=0 }^{ a-1 }\dfrac{2}{|V(K_k)|^2} \\
	            &\leq   g + \sum_{ k=0 }^{ a-1 }\dfrac{2}{\left(|V(K_{a-1})| + (a-1-k)\right)^2} \\ 
	            &<      g + \sum_{i = |V(K_{a-1})|}^\infty\dfrac{2}{i^2} \\
	            &<      g + \int_{i = |V(K_{a-1})|-1}^\infty \dfrac{2}{x^2} \, dx \\
	            &=      g + \frac{2}{|V(K_{a-1})|-1} \\
	            &\leq      g + \frac{2}{\lfloor 4/(\varepsilon g)\rfloor} .
    \end{align*}
    Since $\varepsilon g < 1$, it is the case that $\lfloor 4/(\varepsilon g)\rfloor > 2/(\varepsilon g)$. Consequently,
    \begin{align}
	    g_a     \leq   (1 + \varepsilon) g . \label{eq:ga}
    \end{align}~\\
    
    Let $b$ be the least index $a\leq b\leq \ell$ so that $\x_b(u)\geq \varepsilon g$ for all $u\in V(K_b)$.  
    Note that $b$ is well-defined since $\x_{\ell}(u) = 1 > \varepsilon g$ where $x$ is the optimal weighting for  $K_\ell$, a CRG with a single vertex.  
    For any $k\in\{a,\ldots,b-1\}$, it is the case that $\x_k(u_k) < \varepsilon g$ and that $\x_k(u_k) \leq 1/|V(K_k)|\leq 1/2$ and again by Proposition \ref{prop:g of sub-CRG}, 
    \begin{align*}
        g_{k+1}     \leq    g_k + \dfrac{\x_k(u_k)^2}{(1-\x_k(u_k))^2}
                    <       g_k + \dfrac{(\varepsilon g)^2}{(1/2)^2}
                    =       g_k + 4\varepsilon^2g^2.  
    \end{align*}
    Finally by~\eqref{eq:ga} and since $b-a\leq |V(K_a)|\leq 4/(\varepsilon g)$, 
    \begin{align*}
        g_b \leq g_a + \sum_{ k=a }^{b-1}4\varepsilon^2g^2
        \leq (1+\varepsilon)g + |V(K_a)|\cdot 4\varepsilon^2g^2
        \leq (1+17\varepsilon)g.  
    \end{align*}
    Letting $K':=K_b$, we have the desired sub-CRG.  
\end{proof}~\\

Proposition~\ref{prop:bounded degree} below implies that we can decrease the degree of the underlying graph of a CRG $K$ without changing $g_K(p)$ too much.  
\begin{prop}\label{prop:bounded degree}
    Fix $p\in (0,1)$ and $\varepsilon\in (0,1)$. If $K$ is a $p$-core CRG,  
    then there exists a $p$-core sub-CRG $K'$ of $K$ so that 
    \begin{enumerate}
        \item $g_{K'}(p) \leq (1+\varepsilon)g_K(p)$, and \item for each $u\in V(K')$, $u$ is incident in $K'$ to at most 
        \begin{align*}
            17\varepsilon^{-1}\cdot\max\left\{\dfrac1{p}, \dfrac1{1-p}\right\}
        \end{align*}
        black or white edges.  
    \end{enumerate}
\end{prop}
\begin{proof}
    We prove the claim for all $p\in (0,1/2]$. By duality (that is, by replacing $K$ with $\overline{K}$ and $p$ with $1-p$) the claim holds also for all $p\in [1/2,1)$.  
    By Proposition \ref{prop:sub-CRG} applied to $K$ and $\varepsilon/17$, there exists a sub-CRG $K'$ of $K$ so that $g_{K'}(p)\leq (1+\varepsilon)g_K$ and whose  optimal weight vector $\x\in\simplex_{K'}$ has $\x(u)\geq \varepsilon g/17$ for all $u\in V(K')$.  
    By Theorem \ref{thm:p-core colors}, the white vertices of $K'$ are incident to no white or black edges.  
    So it suffices to prove the desired inequality for black vertices.  
    Suppose $u\in \vb(K')$.  
    By Theorem \ref{thm:degree}, $M_{K'}(p)\x = g\one$ and it follows that 
    \begin{align*}
        \dfrac{g}{p}    >   \dfrac{g}{p} - \dfrac{1-p}p\,\x(u)
                        =   \sum_{uv\in\ew(K')}\x(v)
        \geq \dfrac{\varepsilon g}{17}\cdot \left|\{v : uv\in\ew(K')\}\right| .  
    \end{align*}
    So the number of vertices adjacent to a vertex via a non-gray edge is at most $17/(\varepsilon p)$, as desired.  
\end{proof}~\\

Next, we uniformly bound the diameter of all $p$-core CRGs, for each $p\in (1/3,2/3)$.
\begin{prop}\label{prop:bounded diameter}
	For all positive integers $d$, the CRG associated to the path $P_d$ on $d$ vertices is $p$-prohibited for all 
	\begin{align}
		p\in \left[\dfrac{1}{1+2\cos(\pi/(d+1))}, 1-\dfrac{1}{1+2\cos(\pi/(d+1))}\right]. \label{eq:prohibited range} 
	\end{align}
\end{prop}
\begin{proof}
    It is a well-known fact from spectral graph theory that the spectrum of the adjacency matrix of $P_d$, the path on $d$ vertices is the multiset $\left\{2\cos\left(\frac{\pi k}{d+1}\right) : k\in\left\{1,\ldots,d\right\}\right\}$. 
    See, for example, ~\cite{CvetkovicRowlinsonSimicGraphSpectra}. 
    In particular, the minimum such eigenvalue is $-2\cos(\pi/(d+1))$.
    Lemma~\ref{lem:eigenvalue prohibited} gives that $P_d$ is $p$-prohibited for $p$ in the stated range.
\end{proof}~\\

Finally, we prove Lemma \ref{lem:bounded components}.  
\begin{proof}[Proof of Lemma~\ref{lem:bounded components}]
    Since the sequence $\left\{\frac{1}{1+2\cos(\pi/(d+1))}\right\}$ is monotone decreasing and converges to $1/3$, there exists a positive integer $d = d_p$ so that $P_d$ is $p$-prohibited.  
    Let $K'$ be the sub-CRG from Proposition~\ref{prop:bounded degree} and write $G$ for its underlying graph.  
    Then by construction, $G$ has degree at most 
    \begin{align*}
        D = 17\varepsilon^{-1}\cdot \max\left\{
            \dfrac1p, \dfrac1{1-p}
        \right\}.  
    \end{align*}
    Let $C$ be any component of $G$.  
    Since $P_d$ is $p$-prohibited, $C$ has diameter at most $d-1$.  
    It follows that 
    \begin{align*}
        |C| \leq 1 + D + D(D-1)+\cdots+D(D-1)^{d-1} =: B(p,\varepsilon), 
    \end{align*}
    as desired.  
\end{proof}~\\

\subsection{Forbidding a random graph}
\label{sub-sec:forbidding a random graph}

Now we proceed to prove our main result, Theorem~\ref{thm:ed of random graph}. First, recall that Theorem \ref{thm:H-chromatic number} says that if $F'\sim\GG(n_0',p_0)$ then \aas, 
\begin{align*}
    \chi_H(F') = (1+o(1))\,c_\hh(p_0)\,\dfrac{n_0'}{2\log_2n_0'}.  
\end{align*}
This is useful because in the proof of Lemma~\ref{lem:B-spectrum of random graph}, we repeatedly decompose induced subgraphs of a random graph of order $n_0\gg n_0'$.  
An essential tool is the following restatement of a theorem of Bollob\'as and Thomason:  
\begin{lem}[Lemma 5.1 from \cite{BollobasThomasonStructureColoring}]\label{lem:c and g}
    Let $K$ be a CRG and define $\hh$ to be the hereditary property of graphs, $G$, such that $G\mapsto K$.  
    For all $p_0\in (0,1)$, 
    \begin{align*}
        c_\hh(p_0)  =   -\log_2(p_0(1-p_0)) \cdot 
                        g_K\left(p^*\right)
    \end{align*}
\end{lem}~\\

\begin{lem}\label{lem:B-spectrum of random graph}
    Let $p_0$ and $\varepsilon$ be fixed such that $p_0\in (0,1)$ and $\varepsilon\in (0,1)$. Moreover, fix $\bb$ to be a finite set of CRGs. If $F\sim\GG(n_0,p_0)$, the following holds \aas~as $n_0\to\infty$:  
    For all CRGs $K$ such that all components of $K$ lie in $\bb$ and $F\not\mapsto K$, then 
    \begin{align}\label{eq:g threshold}
        (1+\varepsilon)\, g_{K}\left(p^*\right)
        \geq\dfrac{2\log n_0}{-\log(p_0(1-p_0))\cdot n_0} .  
    \end{align}
\end{lem}
\begin{proof}
    For any function $\mu:\bb\to\{0,1,\dots\}$, define the CRG, $K_{\mu}$ as follows:
    \begin{align*}
        K_\mu   :=  \bigoplus_{B\in\bb}\mu(B)\cdot B.  
    \end{align*}
    That is, $K_{\mu}$ consists of a disjoint union of $\mu(B)$ copies of $B$, for all $B\in\bb$.
    For any induced subgraph $G$ of the random graph $F\sim\GG(n_0,p_0)$ and any CRG $K$, we will denote the event that $G$ embeds into $K$ by $[G\mapsto K]$.  
    Additionally, let 
    \begin{align*}
        E_{\mu}     :=  [F\mapsto K_\mu].  
    \end{align*}~\\
 
     Let
     \begin{align*}
        \bb_0   :=  \left\{ \mu:\bb\to\{0,1,\dots\} : 
                            (1+\varepsilon) g_{K_\mu}(p^*) < \dfrac{2\log n_0} {-\log(p_0(1-p_0))\cdot n_0} 
                    \right\}.  
    \end{align*}
    To prove the desired claim, it is equivalent to show that the probability that $F\not\mapsto K_{\mu}$ for all $\mu\in\bb_0$ goes to zero. That is, 
    \begin{align}
        \lim_{n_0\to\infty} \PP\left[ \bigcup_{\mu\in\bb_0} \overline{E_{\mu}} \right]
        = 0. \label{eq:singleunion}
    \end{align}~\\
    
    Recall $F\sim\GG(n_0,p_0)$. We will partition $V(F)$ by setting  
    \begin{align*}
        C   :=  \left\lceil 2|\bb|\cdot \dfrac{1+\varepsilon/2}{\varepsilon/2}\right\rceil
    \end{align*}
    and let $I_1,\ldots,I_C$ be an equipartition of $V(F)$, i.e., $|I_k|\in\{\lfloor n_0/C\rfloor,\lceil n_0/C\rceil\}$ for $k\in \{1,\ldots,C\}$. Let $n_0'=\lceil n_0/C\rceil$. \\

    For any $B\in\bb$, set 
    \begin{align*}
        m_B     :=  \left\lfloor \dfrac{(1+\varepsilon/2)\cdot (n_0/C) \cdot \left(-\log(p_0(1-p_0))\right) \cdot g_B(p^*)}{2\log (n_0/C)} \right\rfloor .
    \end{align*}
    We expect to be able to $(m_B\cdot B)$-color a $\GG(n_0', p_0)$ graph (hence a $\GG(n_0'-1, p_0)$ graph as well).  
    Now for any $k\in\{1,\ldots,C\}$ and any $B\in\bb$, we define the event 
    \begin{align*}
        E_{k,B}     :=  [F[I_k]\mapsto m_B\cdot B] . 
    \end{align*}
    In other words, $E_{k,B}$ is the event that the subgraph of $F$ that is induced by vertices in $I_k$ is colorable by $m_B$ copies of the CRG $B$. \\

    The induced subgraphs $F[I_1],\dots,F[I_C]$ are each independently sampled according to the Erd\H{o}s-R\'{e}nyi random graph model $\GG(n_0',p_0)$.  
    Moreover, the number of events of the form $E_{k,B}$ is equal to $C\cdot |\bb|$, which is bounded as a function of the constants $p_0$, $\varepsilon$, and $|\mathcal{B}|$. \\

    Then by Theorem \ref{thm:H-chromatic number} and Lemma \ref{lem:c and g}, since the number of vertices in each $I_k$ uniformly tends to $\infty$, 
    \begin{align}
        \lim_{n_0\to\infty} \PP\left[\bigcup_{k\in\{1,\ldots,C\}} \bigcup_{B\in\bb} \overline{E_{k,B}}\right]
         = 0. \label{eq:doubleunion} 
    \end{align}~\\
    
    It suffices to show that  
    \begin{align*}
        \bigcup_{\mu\in\bb_0} \overline{E_{\mu}} 
        \subseteq \bigcup_{k\in\{1,\ldots,C\}} \bigcup_{B\in\bb} \overline{E_{k,B}} ,  
    \end{align*}
    because then \eqref{eq:doubleunion} will imply \eqref{eq:singleunion} and complete the proof. \\
    
    Indeed, suppose $\mu\in \bb_0$ and suppose $\varphi_{k,B}$ are embeddings defined by the events $E_{k,B}$, for all $k\in \{1,\ldots,C\}$ and all $B\in\bb$.  
    Further, for any $B\in\bb$, let 
    \begin{align*}
        M_B     :=  \left\lfloor\dfrac{\mu(B)}{m_B}\right\rfloor .
    \end{align*}~\\
    
    We will show that if $\sum_{B\in\bb} M_B\geq C$, then $F$ can be colored by $M_B$ copies of $m_B\cdot B$, over all $B\in\bb$, and so $F\mapsto K_{\mu}$. 
    We begin by summing the $M_B$'s.  
    \begin{align*}
        \sum_{B\in\bb} M_B
        &=      \sum_{B\in\bb} \left\lfloor\dfrac{\mu(B)}{m_B}\right\rfloor\\
        &\geq   \sum_{B\in\bb} \left(\dfrac{\mu(B)}{m_B}-1\right) \\
        &=      -|\bb| + \sum_{B\in\bb} \dfrac{\mu(B)\cdot 2\log(n_0/C)}{(1+\varepsilon/2)\cdot (n_0/C) \cdot \left(-\log(p_0(1-p_0))\right)\cdot g_B(p^*)} \\
        &=      -|\bb| + \dfrac{2C\log(n_0/C)}{(1+\varepsilon/2)n_0}
        \cdot \sum_{B\in\bb}
            \dfrac{\mu(B)}{-\log(p_0(1-p_0))\cdot g_B(p^*)} .
    \end{align*}
    By Proposition~\ref{prop:archapelago}, 
    \begin{align*}
        \sum_{B\in\bb} M_B
        &=      -|\bb| + \dfrac{2C\log(n_0/C)}{(1+\varepsilon/2)n_0}
        \cdot \dfrac{1}{-\log(p_0(1-p_0))\cdot g_{K_\mu}(p^*)} .
    \end{align*}
    By~\eqref{eq:g threshold},
    \begin{align*}
        \sum_{B\in\bb} M_B
        &\geq   -|\bb| + \dfrac{2C\log(n_0/C)}{(1+\varepsilon/2)n_0}
        \cdot \dfrac{(1+\varepsilon)n_0}{2\log n_0}
        \\
        &=      -|\bb| + C \left(1 - \frac{\log C}{\log n_0}\right) \left(1 + \frac{\varepsilon}{2+\varepsilon}\right) \\
        &=      C + C \left(\frac{\varepsilon}{2+\varepsilon} - \frac{2+2\varepsilon}{2+\varepsilon}\cdot\frac{\log C}{\log n_0}\right) - |\bb| 
    \end{align*}~\\
    
    With $\varepsilon$ fixed and $n_0 \gg C \gg |\bb|$, we have $\sum_{B\in\bb} M_B \geq C$, as desired.  
    Thus, the index set $\{1,\ldots,C\}$ may be partitioned as 
    \begin{align*}
        \{1,\ldots,C\}  =   \bigudot_{B\in\bb} S_B
    \end{align*}
    so that $|S_B|\leq M_B$, for each $B\in\bb$.  
    Also, for each $B\in\bb$, we combine the embeddings $\varphi_{B,k}$ for all $k\in S_B$ to form an embedding
    \begin{align*}
        F\left[\bigudot_{k\in S_B}I_k\right]
        \mapsto (|S_B|\cdot m_B)\cdot B\subseteq \mu(B)\cdot B . 
    \end{align*}
    Combining all such embeddings, $F\mapsto K_\mu$, as desired.  
    So 
    \begin{align*}
        \bigcap_{k\in\{1,\ldots,C\}} \bigcap_{B\in\bb} E_{k,B}
        \subseteq \bigcap_{\mu\in\bb_0} E_{\mu}
        \qquad\Longleftrightarrow\qquad
        \bigcup_{\mu\in\bb_0} \overline{E_{\mu}}
        \subseteq \bigcup_{k\in\{1,\ldots,C\}} \bigcup_{B\in\bb} \overline{E_{k,B}} . 
    \end{align*}
    This completes the proof of the desired claim.  
\end{proof}~\\

Finally, we are ready to prove Theorem \ref{thm:ed of random graph}.~\\  

\subsection{Proof of Theorem~\ref{thm:ed of random graph}}
\label{sub-sec:ed of random graph}
    Formally, our goal is to show that for each $\varepsilon\in(0,1)$, the following occurs~\aas~as $n_0\rightarrow\infty$: 
    \begin{align}
        \sup_{p\in I}\left|\ed_{\hh}(p) \left(\dfrac{2\log n_0}{n_0} \cdot 
            \min\left\{\dfrac{p}{-\log(1-p_0)}, \dfrac{1-p}{-\log p_0}\right\}\right)^{-1} - 1\right|<\varepsilon , \label{eq:formal statement}
    \end{align}
    where 
    \begin{align*}
        I   =   \left\{     \begin{array}{rl}
                                \left[0,1\right] ,    & \mbox{if $p^*\in [\phirange]$;} \\
                                \left[1/3,1\right] ,  & \mbox{if $p^*\in [0,1-\varphi^{-1})$;} \\
                                \left[0,2/3\right] ,  & \mbox{if $p^*\in (\varphi^{-1},1]$.}
                            \end{array}\right.
    \end{align*} \\
    
    We first establish an upper bound for $\ed_{\hh}(p)$ for all $p\in [0,1]$. 
    By the main result in~\cite{BollobasChromatic}, if $F\sim\GG(n_0,p_0)$, then \aas~as $n_0\rightarrow\infty$, 
    \begin{align*}
        \chi(F)-1       &\geq  (1-\varepsilon/2)\; \dfrac{n_0}{2 \log_{1/(1-p_0)} n_0} , \\ 
        \chi(F^c)-1     &\geq  (1-\varepsilon/2)\; \dfrac{n_0}{2 \log_{1/p_0} n_0} .
    \end{align*}
    Clearly $F$ does not map into $\chi(F)-1$ white vertices or $\chi(F^c)-1$ black vertices.  
    By, for example Theorem \ref{thm:inf = min}, and the fact that $1/(1-\varepsilon/2) < 1+\varepsilon$, then the following occurs~\aas:
    \begin{align}
        \ed_{\hh}(p) 
        &   \leq (1+\varepsilon)\; \min\left\{\dfrac{p}{n_0/(2 \log_{1/(1-p_0)} n_0)}, \dfrac{1-p}{n_0/(2 \log_{1/p_0} n_0)}\right\} \nonumber \\
        &   = (1+\varepsilon)\; \dfrac{2\log n_0}{n_0} \cdot 
            \min\left\{\dfrac{p}{-\log(1-p_0)}, \dfrac{1-p}{-\log p_0}\right\} .  \label{eq:upper bound on ed of random graph}
    \end{align}~\\
    
    We now find a lower bound to match Inequality~\eqref{eq:upper bound on ed of random graph} over the interval $I$ as stated above.  We will do this by finding the lower bound for $p\in (1/3,2/3)$ and then use concavity show how this extends to $I$ in the various cases. \\ 
    
    We will choose a $\pt\in (1/3,2/3)$ depending on the case: 
    \begin{align*}
        \pt     &:=     \left\{ \begin{array}{rl}
                                p^*,                & \mbox{if $p^*\in (1/3,2/3)$;} \\
                                1/3+\varepsilon/9,  & \mbox{if $p^*\in (0,1/3]$;} \\
                                2/3-\varepsilon/9,  & \mbox{if $p^*\in [2/3,1)$.}
                                \end{array} \right.
    \end{align*}
    Let $K = K(\pt)\in\kk_\hh$ be a $\pt$-core CRG that satisfies $\ed_{\hh}(\pt) = g_K(\pt)$, as guaranteed by Theorem~\ref{thm:inf = min}. \\ 

    Since $\pt\in (1/3,2/3)$,  Lemma~\ref{lem:bounded components} gives that there exists a sub-CRG $K' = K'(\pt,\varepsilon/4)$ of $K$ so that 
    \begin{align}
        g_K(\pt)  \leq    g_{K'}(\pt)   \leq    (1+\varepsilon/4) \, g_K(\pt) \label{eq:K prime sandwich}
    \end{align}
    and whose components lie in some finite set $\bb = \bb(\pt,\varepsilon/4)$ of CRGs. \\  
    
    The function $g_{K'} : [0,1]\to [0,1]$ is concave-down.
    To see this, let $M_{K'}(p)$ be the matrix defined by the CRG $K$. Let $p_1,p_2\in [0,1]$, $t\in [0,1]$, and let $\x\in\Delta_{K'}$ be the vector that witnesses the value of $g_{K'}$ at $t p_1 + (1-t) p_2$,
    \begin{align*}
        g_{K'}\left(t p_1 + (1-t) p_2\right) 
        &=      \langle \x, M_{K'}\left(t p_1 + (1-t) p_2\right) \x\rangle \\
        &=      t\cdot\langle\x, M_{K'}(p_1)\x\rangle + (1-t)\cdot\langle \x, M_{K'}(p_2) \x\rangle \\
        &\geq   t\cdot g_{K'}(p_1) + (1-t)\cdot g_{K'}(p_2) ,
    \end{align*}
    establishing the concavity of $g_{K'}$. 
    Since $g_{K'}(0),g_{K'}(1)\geq 0$, the graph of $g_{K'}$ lies above the line segment from $(0,0)$ to $(\pt,g_{K'}(\pt))$ to $(1,0)$.  
    So 
    \begin{align}\label{eq:concavity lower bound}
        g_{K'}(p) 
        \geq 
             g_{K'}(p^*)\cdot 
             \min\left\{\dfrac{p}{p^*}, \dfrac{1-p}{1-p^*}\right\}
    \end{align}

    Also by Lemma \ref{lem:B-spectrum of random graph} (recall the definition of $p^*$ from~\eqref{eq:p star}), the following is true~\aas:
    \begin{align}
        g_{K'}(p^*)     \geq    (1-\varepsilon/4)\cdot\dfrac{2\log n_0}{-\log(p_0(1-p_0)) \cdot n_0}. \label{eq:K prime p star}  
    \end{align}~\\
    
    Combining~\eqref{eq:K prime sandwich},~\eqref{eq:concavity lower bound}, and~\eqref{eq:K prime p star},
    \begin{align}
        \ed_\hh(\pt) 
        &=      g_K(\pt) \nonumber \\
        &\geq   \frac{1}{1+\varepsilon/4}\, g_{K'}(\pt) \nonumber \\
        &\geq   \frac{1}{1+\varepsilon/4}\, g_{K'}(p^*)\cdot \min\left\{\dfrac{\pt}{p^*}, \dfrac{1-\pt}{1-p^*}\right\} \nonumber \\
        &\geq   \frac{1-\varepsilon/4}{1+\varepsilon/4}\cdot \frac{2\log n_0}{-\log (p_0(1-p_0))\cdot n_0}\cdot \min\left\{\dfrac{\pt}{p^*}, \dfrac{1-\pt}{1-p^*}\right\} \nonumber \\
        &\geq   (1-\varepsilon/2)\, \frac{2\log n_0}{-\log (p_0(1-p_0))\cdot n_0}\cdot \min\left\{\dfrac{\pt}{p^*}, \dfrac{1-\pt}{1-p^*}\right\} 
        \label{eq:ed p lower bound min}
    \end{align}~\\

    \noindent\textbf{Case 1.} $p_0\in (\phirange)$ $\Longleftrightarrow$ $p^*\in (1/3,2/3)$. \\
    \indent In this case, $\pt = p^*$.  
    By~\eqref{eq:ed p lower bound min} and the concavity of $\ed_{\hh}(p)$, 
    \begin{align*}
        \ed_{\hh}(p^*) 
        &\geq   (1-\varepsilon/2)\, \frac{2\log n_0}{-\log (p_0(1-p_0))\cdot n_0} \\
        \ed_{\hh}(p) 
        &\geq   (1-\varepsilon/2)\, \frac{2\log n_0}{-\log (p_0(1-p_0))\cdot n_0}\cdot \min\left\{\dfrac{p}{p^*}, \dfrac{1-p}{1-p^*}\right\} \\
        &=   (1-\varepsilon/2)\, \frac{2\log n_0}{n_0}\cdot \min\left\{\dfrac{p}{-\log (1-p_0)}, \dfrac{1-p}{-\log p_0}\right\} .
    \end{align*}
    This satisfies~\eqref{eq:formal statement} for all $p\in [0,1]$, completing the proof in this case. \\

    \noindent\textbf{Case 2.} $p_0\in (0,1-\varphi^{-1}]$ $\Longleftrightarrow$ $p^*\in (0,1/3]$. \\
    \indent In this case, $\pt = 1/3 + \varepsilon/9 > p^*$.  
    By~\eqref{eq:ed p lower bound min} and the concavity of $\ed_{\hh}$, 
    \begin{align*}
        \ed_{\hh}(\pt) 
        &\geq   (1-\varepsilon/2)\, \frac{2\log n_0}{-\log (p_0(1-p_0))\cdot n_0}\cdot \dfrac{1-\pt}{1-p^*} \\
        \ed_{\hh}(p) 
        &\geq   (1-\varepsilon/2)\, \frac{2\log n_0}{-\log (p_0(1-p_0))\cdot n_0}\cdot \dfrac{1-\pt}{1-p^*}\cdot \min\left\{\dfrac{p}{\pt}, \dfrac{1-p}{1-\pt}\right\} .
    \end{align*}~\\
    
    Substituting $\pt=1/3+\varepsilon/9$ and $p=1/3$,
    \begin{align*}
        \ed_{\hh}(1/3) 
        &\geq   (1-\varepsilon/2)\, \frac{2\log n_0}{-\log (p_0(1-p_0))\cdot n_0}\cdot \dfrac{2/3-\varepsilon/9}{1-p^*}\cdot \dfrac{1/3}{1/3+\varepsilon/9} \\
        &=      \dfrac{(1-\varepsilon/2)(1/3-\varepsilon/18)}{1/3+\varepsilon/9}\cdot \frac{2\log n_0}{n_0}\cdot \dfrac{2/3}{-\log p_0} \\
        &\geq       (1-\varepsilon)\, \frac{2\log n_0}{n_0}\cdot \dfrac{2/3}{-\log p_0} .
    \end{align*}
    
    Again by concavity, 
    \begin{align*}
        \ed_{\hh}(p) 
        \geq   (1-\varepsilon)\, \frac{2\log n_0}{n_0}\cdot \dfrac{2/3}{-\log p_0}\cdot \min\left\{\dfrac{p}{1/3},\dfrac{1-p}{2/3}\right\} .
    \end{align*}
    This matches the upper bound~\eqref{eq:upper bound on ed of random graph} for all $p\in [1/3,1]$ and, in fact, if $p^*=1/3$, then it matches the upper bound for all $p\in [0,1]$. 
    This completes the proof in this case.   \\
    
    \noindent\textbf{Case 3.} $p_0\in [\varphi^{-1},1)$ $\Longleftrightarrow$ $p^*\in [2/3,1)$. \\
    \indent This case may be shown with a similar argument as Case 2.  
    In this case, $\pt = 2/3-\varepsilon/9 < p^*$. 
    By \eqref{eq:ed p lower bound min} and the concavity of $\ed_\hh$, 
    \begin{align*}
        \ed_{\hh}(p) 
        \geq   (1-\varepsilon)\, \frac{2\log n_0}{n_0}\cdot \dfrac{1/3}{-\log p_0}\cdot \min\left\{\dfrac{p}{2/3},\dfrac{1-p}{1/3}\right\} . 
    \end{align*}
    This matches the upper bound~\eqref{eq:upper bound on ed of random graph} for all $p\in [0,2/3]$ and if $p^*=2/3$, then it matches the upper bound for all $p\in[0,1]$. This completes the proof in this case and the proof of Theorem~\ref{thm:ed of random graph}. \qed \\
\section{Discussion}
\label{sec:discussion}

In the process of proving the main result, Theorem~\ref{thm:ed of random graph}, we have developed a number of observations that apply generally to computing edit distance functions. Lemma~\ref{lem:eigenvalue prohibited} gives a general condition for which a CRG is $p$-prohbited and Lemma~\ref{lem:dalmatian p-core} shows that for $p\in [\phirange]$, the only CRGs that need to be considered are dalmatian sets and their complements. \\

In this section, we discuss some other general results.

\subsection{Defining the edit distance function with a finite set of CRGs}

The following was conjectured by the first author:
\begin{conj}[\cite{MartinSurvey}]\label{conj:finite set of CRGs}
    Let $\hh$ be a nontrivial hereditary property. For every $\varepsilon>0$ there exists a finite set of CRGs $\kk'=\kk'(\varepsilon,\hh)$ such that 
    \begin{align*}
        \ed_{\hh}(p)    =   \min\left\{g_K(p) : K\in\kk'\right\}    ,&&\mbox{for all $p\in (\varepsilon,1-\varepsilon)$.}
    \end{align*}
\end{conj}~\\

Note that this is stronger than the Marchant-Thomason result in Theorem~\ref{thm:inf = min} which says that, for every $p$ there is a finite set of CRGs that define $\ed_{\hh}(p)$. Conjecture~\ref{conj:finite set of CRGs} asserts that a single finite set will define $\ed_{\hh}(p)$ for all $p$ an arbitrary open interval in $(0,1)$. In Theorem~\ref{thm:finite set of CRGs}, we provide a partial answer by showing that the conjecture is true for $\varepsilon\geq 1-\varphi^{-1}$. \\

\begin{thm}\label{thm:finite set of CRGs}
    Let $\hh$ be a nontrivial hereditary property. There exists a finite set of CRGs, $\kk'=\kk'(\hh)$, such that 
    \begin{align*}
        \ed_{\hh}(p)    =   \min\left\{g_K(p) : K\in\kk'\right\} ,  &&\mbox{for all $p\in [\phirange]$.}
    \end{align*}
\end{thm}

\begin{proof}
     By Lemma~\ref{lem:dalmatian p-core}, the only $p$-core CRGs are denoted $\dd_p$ and consist of components which are dalmatian CRGs if $p\in [\phirangeleft)$, complements of dalmatian CRGs if $p\in (\phirangeright]$, and CRGs with only gray edges if $p=1/2$. \\
    
    It is easy to see that if $p=1/2$ and $\hh=\forb(\ff)$, then for any $F\in\ff$ such that $F\not\mapsto K$, the number of vertices of $K$ is bounded by $\chi(F) + \chi(F^c)$, so the number of such CRGs is finite. 
    We will now show that a finite set of CRGs suffice for $p\in[\phirangeleft)$. The case $p\in(\phirangeright]$ follows by symmetry. Note that $\dd_p$ is the same for all $p\in[0,1/2)$ (that is, CRGs whose components are all dalmatian CRGs) so we denote $\dd_{0,\hh} := \dd_p\cap\kk_{\hh}$. \\
    
    Write $\hh = \forb(\ff)$ and, for a contradiction, let $\{p_k\}_{k=1}^{\infty}\subset[\phirangeleft)$ be an infinite set and let $\kk' := \{K_k\}_{k=1}^{\infty}\subset\dd_{0,\hh}$ an infinite set of CRGs such that $K_k$ is $p_k$-core and $g_{K_k}(p_k)=\ed_{\hh}(p_k)$. \\
 
    For all $k\geq 1$, with $D_i$ denoting a dalmatian CRG of order $i$, we may write 
    \begin{align}
        K_k     =   D_{c_k^{(1)}}\oplus \cdots \oplus D_{c_k^{(\ell_k)}} \oplus (w_k\cdot D_\infty) \label{eq:Kktuple}
    \end{align}
    where $w_k,\ell_k, c_k^{(1)},\ldots,c_k^{(\ell_k)}$ are nonnegative integers and $c_k^{(1)}\geq \cdots \geq c_k^{(\ell_k)}$. \\

    For all $k\geq 1$, $w_k+\ell_k\leq |V(F)|$ for any $F\in\ff$ because otherwise $F\mapsto K_k$ by embedding each vertex of $F$ into a different component of $K_k$.  
    Thus, $\ell_k$ is bounded by an absolute constant $\ell=\ell(\hh)$.
    So we associate each $K_k$ in \eqref{eq:Kktuple} with the $(\ell+1)$-tuple 
    \begin{align*}
        \left(c_k^{(1)},\ldots,c_k^{(\ell)};w_k\right) ,
    \end{align*}
    where $c_k^{(\ell_k+1)}=\cdots=c_k^{(\ell)}=0$ if $\ell>\ell_k$. \\
    
    Because $\kk'$ is infinite, there exists a maximum $m\in\{1,\ldots,\ell\}$ such that $\sup_k \{c_k^{(1)}\}=\cdots=\sup_k \{c_k^{(m)}\}=\infty$. 
    That is, if $m<\ell$, then $\sup_k \{c_k^{(m+1)}\}<\infty$. 
    Thus, there is a fixed (possibly empty) tuple $\left(c_*^{(m+1)},\ldots,c_*^{(\ell)};w_*\right)$ and an infinite subsequence $k_1,k_2,\ldots$ such that $K_{k_i}$ is associated with $(\ell+1)$-tuple  $\left(c_{k_i}^{(1)},\ldots,c_{k_i}^{(m)},c_*^{(m+1)},\ldots,c_*^{(\ell)};w_*\right)$. \\
    
    With this choice of $\left(c_*^{(m+1)},\ldots,c_*^{(\ell)};w_*\right)$, if $\ell'$ is the largest entry such that $c_*^{(\ell')}\geq 1$, then define
    \begin{align*}
        K_*     =   D_{c_k^{(m+1)}} \oplus \cdots \oplus D_{c_k^{(\ell')}} \oplus ((m + w_*)\cdot D_\infty) .
    \end{align*}
    We claim that $K_*\in\dd_{0,\hh}\subseteq\kk_\hh$. \\
    
    If not, then there exists some $F\in\ff$ and some embedding $\phi:V(F)\to V(K_*)$.  
    Let $A_1,\dots,A_m$ be the preimages of the first $m$ copies of $D_\infty$ under $\varphi$, respectively.  
    Then $A_1,\dots,A_m$ are independent sets in $F$. \\

    Let $k$ be sufficiently large so that $c_1^k,\dots,c_m^k\geq |V(F)|$. 
    Then we define $\phi':V(F)\to V(K_k)$ by instead sending the vertices of $A_i$ to distinct vertices of the dalmatian set $D_{c_i^k}$, for each $i\in\{1,\ldots,m\}$.  
    As a result, $\phi'$ is an embedding of $F$ into $K_k$, a contradiction. \\
    
    Finally by Proposition~\ref{prop:archapelago} and Remark~\ref{rem:dalmatian p-core}, for all $p\in (0,1/2)$ and $k$ chosen as above, 
    \begin{align*}
        g_{K_k}(p)^{-1}     &=  \dfrac{w}{p} + \sum_{i=1}^m\dfrac{1}{p+(1-2p)/c_i^k} + \sum_{i=1}^{\ell-m}\dfrac{1}{p+(1-2p)/c_i} \\
                            &<  \dfrac{w+m}{p} + \sum_{i=1}^{\ell-m}\dfrac{1}{p+(1-2p)/c_i} \\
                            &=  g_{K_*}(p)^{-1}
    \end{align*}
    The fact that $g_{K_*}(p_k) < g_{K_k}(p_k) = \ed_{\hh}(p_k)$ contradicts $K_*\in\kk_{\hh}$, hence the original assumption that $\kk'$ is infinite. 
    
\end{proof}

\subsection{Paths}

In Proposition~\ref{prop:bounded diameter}, it is established that $P_d$ is $p$-prohibited for $p\in\left[\frac{1}{1+2\cos(\pi/(d+1))}, 1-\frac{1}{1+2\cos(\pi/(d+1))}\right]$. 
In the case of $d=3$, $P_3$ is $p$-prohibited for $p$ in the interval
\begin{align*}
    [\sqrt{2}-1,2-\sqrt{2}]     \approx    [0.414214,0.585786] .
\end{align*}
However, Lemma~\ref{lem:dalmatian p-core} establishes that $P_3$ is $p$-prohibited if and only if $p$ is in the interval
\begin{align*}
    [1-\varphi^{-1},\varphi^{-1}] 
    =          \left[\frac{3-\sqrt{5}}{2}, \frac{\sqrt{5}-1}{2}\right]
    \approx    [0.381966,0.618034] .
\end{align*}
We ask whether $P_d$ is $p$-prohibited over a larger interval than given in~\eqref{eq:prohibited range}. See Table~\ref{tab:paths} for small values.

\begin{table}[ht]
    \centering
    \begin{tabular}{|r||rl|rl|}\hline
        $d$     &   \multicolumn{2}{c|}{$\left(1+2 \cos(\pi/(d+1))\right)^{-1}$}         &   \multicolumn{2}{c|}{$1-\left(1+2 \cos(\pi/(d+1))\right)^{-1}$} \\ \hline\hline
        $3$     &   $\sqrt{2}-1$      &   $\approx 0.414214$      &   $2-\sqrt{2}$        &   $\approx 0.585786$ \\ \hline
        $4$     &   $(3-\sqrt{5})/2$  &   $\approx 0.381966$      &   $(\sqrt{5}-1)/2$    &   $\approx 0.618034$ \\ \hline
        $5$     &   $(\sqrt{3}-1)/2$  &   $\approx 0.366025$      &   $(3-\sqrt{3})/2$    &   $\approx 0.633975$ \\ \hline 
        $6$     &   &   $\approx 0.356896$  &   &   $\approx 0.643104$ \\ \hline
        $7$     &   &   $\approx 0.351153$  &   &   $\approx 0.648847$ \\ \hline
        $8$     &   &   $\approx 0.347296$  &   &   $\approx 0.652704$ \\ \hline
        $9$     &   &   $\approx 0.344577$  &   &   $\approx 0.655423$ \\ \hline
        $10$    &   &   $\approx 0.342585$  &   &   $\approx 0.657415$ \\ \hline
        $11$    &   &   $\approx 0.341081$  &   &   $\approx 0.658919$ \\ \hline
        $12$    &   &   $\approx 0.339918$  &   &   $\approx 0.660082$ \\ \hline
        $13$    &   &   $\approx 0.339000$  &   &   $\approx 0.661000$ \\ \hline
        $14$    &   &   $\approx 0.338261$  &   &   $\approx 0.661739$ \\ \hline
        $15$    &   &   $\approx 0.337659$  &   &   $\approx 0.662341$ \\ \hline 
    \end{tabular}
    \caption{Table for endpoints of an interval where $P_d$ is prohibited.}
    \label{tab:paths}
\end{table}

\begin{quest}
    For $d\geq 4$, what is the largest interval over which $P_d$ is $p$-prohibited?
\end{quest}~\\

\section{Questions and future work}\label{sec:questions}

\subsection{$p$-core CRGs}\label{sub-sec:p-core classification}

Lemma~\ref{lem:dalmatian p-core} classifies all $p$-core CRGs on the interval $[1-\varphi^{-1},\varphi^{-1}]$.  
\begin{quest}\label{quest:p-core interval}
    For which $a\in (0,1-\varphi^{-1})$ does there exist an elementary classification of all $p$-core CRGs for all $p\in [a,1-a]$?  
    Additionally, are all sufficiently large connected $p$-core CRGs either dalmatian CRGs (if $p\leq 1/2$) or the complement of a dalmatian CRG (if $p\geq 1/2$)?  
\end{quest}~\\

A crucial part of the proof of Theorem~\ref{thm:ed of random graph} is Lemma~\ref{lem:bounded components}, which establishes that, for $p\in (1/3,2/3)$, a $p$-core CRG can be approximated so that the $g$ function does not increase by much, but the components are bounded.

\begin{quest}
    Does Lemma \ref{lem:bounded components} hold if the interval $(1/3,2/3)$ is widened to $(a,1-a)$ for some $a \in (0,1/3)$?  
\end{quest}~\\

\subsection{Inhomogeneous random graphs}

Since the development of graph limits, inhomogeneous generalizations $\GG(n,W)$ of the Erd\H{o}s-R\'{e}nyi random graph models have emerged as a topic of research interest (see~\cite{LovaszGraphLimits}).  
Here, $W:\Omega^2\to[0,1]$ is a {\em graphon}, which is a symmetric measurable function where $\Omega$ is a probability space, frequently $[0,1]$ equipped with the Lebesgue measure.  
To form a $W$-random graph $G\sim\GG(n,W)$, sample $n$ elements $x_1,\dots,x_n\sim \Omega$ independently and form a graph on $\{1,\ldots,n\}$ by adding edge $ij$ independently with probability $W(x_i,x_j)$.  
We may also generate a sequence of $W$-random graphs $(G_n)_{n=1}^\infty\sim\GG(\N,W)$ adding the vertices corresponding to $x_1,x_2,\ldots$, one at a time.  \\

There are several questions we may ask related to the edit distance problem and inhomogeneous random graphs.  \\

First, note that Theorem \ref{thm:ed of random graph} implies that with 
$p_0\in [\phirange]$
and $(F_n)\sim \GG(\N,p_0)$, then \aas, 
\begin{align}\label{eq:g function ratio}
    \lim_{n\to\infty}\sup_{p\in [0,1]}\dfrac{\ed_{\forb(F_{n+1})}(p)}{\ed_{\forb(F_n)}(p)}
    = 1.  
\end{align}~\\

\begin{quest}
    For what graphons $W$ does Equation \eqref{eq:g function ratio} \aas~hold for $(F_n)\sim\GG(\N,W)$?  
\end{quest}
For the next question, we note the following expression for the distance from a homogeneous random graph to a hereditary property.  
Combining Theorems \ref{thm:distance to random graph} and \ref{thm:inf = min}, we see that if $\hh$ is a (fixed) nontrivial hereditary property and $p\in [0,1]$, then with $G_n\sim\GG(n,p)$,  
\begin{align}\label{eq:distance as a min of CRGs}
    \lim_{n\to\infty}\Exp[\dist(G_n,\hh)]
    = \inf_{K\in \kk_\hh}g_K(p)
    = \min_{K\in \kk_\hh}g_K(p).  
\end{align}~\\

\begin{quest}
    If instead $G_n\sim\GG(n,W)$, is there a similar expression for $\limsup_{n\to\infty} \Exp[\dist(G_n,\hh)]$?  
    In particular, can we extend the functions $g_K(\cdot)$ from $[0,1]$ to the set of all graphons so that Equation \eqref{eq:distance as a min of CRGs} holds?  
\end{quest}~\\

\subsection{Acknowledgements}
The authors would like to thank Jan Hladk\'y and Josh Cooper for helpful comments.

\bibliography{mybib}
\bibliographystyle{plain}
\end{document}